%% file: PostprocessingOperator.tex
\documentclass[11pt]{amsart}

\usepackage[utf8]{inputenc}
\usepackage[T1]{fontenc}

\usepackage{amsmath,amssymb,amsfonts,textcomp,amsthm,xifthen,graphicx,color,pgfplots}

\usepackage{enumerate}
\usepackage{enumitem}
\usepackage{soul}
\usepackage{subcaption}
\usepackage{graphicx}
\usepackage[ruled,linesnumbered]{algorithm2e}

\usepackage{pgfplots}
\usepgfplotslibrary{groupplots}
\usepgfplotslibrary{colorbrewer}
\pgfplotsset{compat=newest}

\usepackage{fullpage}

\usepackage[utf8]{inputenc}

\usepackage[pdftex,
  pdfauthor={Thomas F\"uhrer and Manuel S\'anchez},
            pdftitle={Quasi-interpolators with application to postprocessing in FEM},
            ]{hyperref}

\input{header}

\begin{document}

\title{Quasi-interpolators with application to postprocessing in finite element methods}
\date{\today}

\author{Thomas F\"uhrer}
\address{Facultad de Matem\'{a}ticas, Pontificia Universidad Cat\'{o}lica de Chile, Santiago, Chile}
\email{tofuhrer@mat.uc.cl}
\author{Manuel A. S\'anchez}
\address{Instituto de Ingenier\'{i}a Matem\'{a}tica y Computacional, Pontificia Universidad Cat\'{o}lica de Chile, Santiago, Chile}
\email{manuel.sanchez@uc.cl}

\thanks{{\bf Acknowledgment.} 
This work was supported by ANID through FONDECYT project 1210391 (TF) and 1221189 (MS), and by Centro Nacional de Inteligencia Artificial CENIA, FB210017, Basal ANID Chile (MS)}

\begin{abstract}
We design quasi-interpolation operators based on piecewise polynomial weight functions of degree less than or equal to $p$ that map into the space of continuous piecewise polynomials of degree less than or equal to $p+1$. We show that the operators have optimal approximation properties, i.e., of order $p+2$.
This can be exploited to enhance the accuracy of finite element approximations provided that they are sufficiently close to the orthogonal projection of the exact solution on the space of piecewise polynomials of degree less than or equal to $p$.
Such a condition is met by various numerical schemes, e.g., mixed finite element methods and discontinuous Petrov--Galerkin methods.
Contrary to well-established postprocessing techniques which also require this or a similar closeness property, our proposed method delivers a conforming postprocessed solution that does not rely on discrete approximations of derivatives nor local versions of the underlying PDE.
In addition, we introduce a second family of quasi-interpolation operators that are based on piecewise constant weight functions, which can be used, e.g., to postprocess solutions of hybridizable discontinuous Galerkin methods. Another application of our proposed operators is the definition of projection operators bounded in Sobolev spaces with negative indices. Numerical examples demonstrate the effectiveness of our approach.
\end{abstract}
\maketitle



\section{Introduction}\label{s1:introduction}

Quasi-interpolation operators are essential tools in the analysis of finite element methods. They prove to be especially useful when nodal interpolants are not well defined. The fundamental idea is averaging over a neighborhood instead of evaluating at specific points. In a priori error analysis, they help derive convergence rates for solutions with reduced regularity, see, e.g.,~\cite[Sec.~4.8]{BrennerScott}. One of their main applications is in a posteriori error estimation, where they are used to localize residuals, see, e.g.~\cite{CarstensenQuasiInt99}.

In the seminal work~\cite{Clement75} Cl\'ement introduced a very important class of quasi-interpolators. 
Since then, many authors have worked on the topic, including~\cite{SZ_90,BernardiGirault98,ErnGuermond17,Oswald93,Melenk05} to name a few. 
The mentioned references deal with positive order Sobolev spaces, while, recently, interest in locally defined quasi-interpolators in Sobolev spaces of negative index has grown, see~\cite{DieningStornTscherpel23} and references therein.

This work introduces and studies a novel class of quasi-interpolators in $H^1$. We briefly discuss their main features next.
Let $\TT$ denote a triangulation, with mesh-size parameter $h$, of the Lipschitz domain $\Omega\subseteq \R^d$, with $d=1,2$, and $P^p(\TT)$ the space of piecewise polynomials of degree $\leq p$, and $P_c^p(\TT) = P^p(\TT)\cap C^0(\Omega)$ the space of continuous piecewise polynomials.
Our quasi-interpolation operators $\qintpz\colon L^2(\Omega)\to P_{c,0}^{p+1}(\TT)=P_c^{p+1}(\TT)\cap H_0^1(\Omega)$ have the form 
\begin{align*}
  \qintpz v = \sum_{\star\in\II_{p+1,0}} \ip{v}{\phi_\star}_\Omega \eta_\star
\end{align*}
where $\II_{p+1,0}$ denotes an index set with each $\star\in\II_{p+1,0}$ corresponding to an interior vertex, interior edge, or element degree of freedom, and $\eta_{\star}$ corresponds to shape functions of the space $P_{c,0}^{p+1}(\TT)$. Here, $(\cdot, \cdot)_{\Omega}$ denotes the $L^2$ inner product over $\Omega$.
While the form of the operator is canonical, the novelty lies in the construction of the weight functions $\phi_\star$. 
We define $\phi_\star\in P^p(\TT)$ such that the operator locally preserves polynomials of degree $\leq p+1$. Consequently, a Bramble--Hilbert argument and local $L^2$ bounds yield
\begin{align*}
  \norm{v-\qintpz v}{\Omega} \lesssim h^{p+2}\norm{D^{p+2}v}\Omega.
\end{align*}

Critical for the existence of weight functions $\phi_\star$ is a property on the intersection of orthogonal polynomials on triangular patches analyzed in~\cite{KoornwinderSauter15}. There, the authors used a weighted inner product and their result has applications in adaptive $hp$-FEM~\cite{BankParsaniaSauter13}.
In this work, we demonstrate that the proof ideas from~\cite{KoornwinderSauter15} carry over to the case when the canonical $L^2$ inner product is used. 
Additionally, ~\cite{MixedFEMHm1} studied a weighted Cl\'ement quasi-interpolator with the same properties for the lowest-order case $p=0$. In Section~\ref{sec:constr:lowest} below, we recall the construction presented there.

One of the main consequences of the design of our operators is that
\begin{align}\label{eq:intro:supconv}
  \norm{u-\qintpz u_\TT}\Omega = \OO(h^{p+2})
\end{align}
provided that $u_\TT\in P^p(\TT)$ is \emph{superclose} to $\Pip u$, where $\Pip$ is the $L^2$ projection onto $P^p(\TT)$, i.e., 
\begin{align*}
  \norm{\Pip u-u_\TT}\Omega = \OO(h^{p+2}).
\end{align*}
The latter is satisfied if $u_\TT$ is the solution of, e.g., mixed finite element methods (mixed FEM)~\cite{Stenberg91} or discontinuous Petrov--Galerkin methods with optimal test functions (DPG) based on ultraweak formulations~\cite{SupConv2}, to name a few.
We stress that~\eqref{eq:intro:supconv} can be interpreted as a \emph{superconvergence} result since $\qintpz = \qintpz \Pip$, i.e., $\qintpz$ ``sees'' only piecewise polynomials of degree $\leq p$ and the approximation estimate $\norm{v-\Pip v}\Omega \lesssim h^{p+1}\norm{D^{p+1}v}\Omega$ is sharp.
Let us note that contrary to other popular (local) postprocessing techniques, like~\cite{Stenberg91,BrambleXu89}, 
our proposed postprocessing scheme does not require the use of an approximate gradient nor the use of a local PDE approximation.

A second family of quasi-interpolation operators, denoted by $\qintzpz$, is introduced, which uses piecewise constant weight functions $\phi_\star$. These operators have similar properties as $\qintpz$, but the existence of weight functions is non-trivial and requires certain additional assumptions on the underlying mesh. These operators render the optimal approximation property \eqref{eq:intro:supconv} provided that a more relaxed supercloseness assumption for the projection onto piecewise constants of the approximation is satisfied. We exemplify a use case for enhancing the accuracy of approximations stemming from a hybridizable discontinuous Galerkin (HDG) method.

Quasi-interpolators in negative order Sobolev spaces have become an important tool in the analysis of multilevel norms/preconditioners~\cite{MultilevelNorms21,StevensonVanVenetie20}, regularization of rough load terms~\cite{VeeserZanotti18,FHK22,MixedFEMHm1}, and convergence rates for space-time methods~\cite{DieningStornTscherpel23}.
To demonstrate another application of the quasi-interpolator $\qintpz$, we introduce projection operators onto $P^p(\TT)$ that are based on $\qintpz$ and study some of their properties. We note that the operators constructed in~\cite{DieningStornTscherpel23} map into the space of piecewise continuous polynomials.

The outline of this article is as follows:
In Section~\ref{sec:notation} we introduce notation, polynomial spaces and some standard assumptions on basis functions. 
Section~\ref{sec:qintpz} introduces the first family of quasi-interpolation operators. We first design the operators for vanishing traces in Section~\ref{sec:def} and then without boundary constraints in Section~\ref{sec:def:general}. Theorems~\ref{thm:qintpz} and~\ref{thm:qintp} show the main results on the approximation properties of these operators. Section~\ref{sec:constr} gives an overview on how to construct the local weight functions $\phi_\star$ and consequently the operators. 
In Section~\ref{sec:qintzpz} we introduce the second family of quasi-interpolators based on piecewise constant weight functions. Their construction is based on Assumption~\ref{ass:vertices:constant}.
The definitions and corresponding main results are found in Section~\ref{sec:def:P0} and~\ref{sec:def:general:P0}.
Applications to postprocessing of finite element solutions are discussed in Section~\ref{sec:appl:postproc}.
Then, Section~\ref{sec:appl:proj} demonstrates an application of the quasi-interpolators to define projection operators onto $P^p(\TT)$ that are uniformly bounded in $H^{-1}(\Omega)$.
Some numerical experiments are provided in Section~\ref{sec:numeric}.
This work is concluded by Section~\ref{sec:conclusion} and
Appendix~\ref{sec:proof2d} provides a proof of the technical Lemma~\ref{lem:fullrank} for $d=2$.

\section{Notations}\label{sec:notation}
In this section, we introduce notations and definitions and collect some basic results on polynomial spaces. 
Throughout, let $\Omega\subset \R^d$ denote a bounded Lipschitz domain with polygonal boundary $(d=2)$ or an open interval $(d=1)$.

\subsection{Triangulation}

Let $\VV$ denote the set of vertices, $\VV_0$ interior vertices, $\EE$ the set of edges, $\EE_0$ interior edges, 
and $\TT$ the set of elements or triangulation of the domain $\Omega$. Let $h_\TT\in L^\infty(\Omega)$ with $h_\TT|_T :=\diam(T)$ be the mesh-size function of the triangulation and set $h:=\norm{h_\TT}{L^\infty(\Omega)}$. Note that for $d=1$ we set $\EE=\emptyset$. 
For each $E\in\EE$ we denote by $\VV_E$ its vertices. Similarly, 
$\VV_T$ stands for vertices of the element $T\in\TT$. 

We next formulate a mild assumption on the mesh.
\begin{assumption}\label{ass:vertices}
  There is at least one interior vertex, i.e., $\VV_0\neq\emptyset$.
\end{assumption}
Note that, for $d=1$, Assumption \ref{ass:vertices} is equivalent to $\#\TT\geq 2$. For $d=2$ our assumption excludes meshes as shown in the left plot of Figure~\ref{fig:mesh1}. The assumption can easily be guaranteed by dividing one element $T$, as shown in the right plot of Figure~\ref{fig:mesh1}.
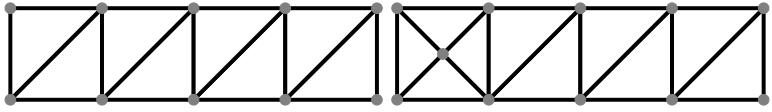
\begin{figure}
  \begin{center}
    \input{FigureMeshes}
  \end{center}
  \caption{Left triangulation does not satisfy Assumption~\ref{ass:vertices}. Right triangulation satisfies Assumption~\ref{ass:vertices} with $R=4$.}
  \label{fig:mesh1}
\end{figure}

For $S\subseteq\overline\Omega$ we denote by $\patch(S) \subseteq \TT$ its element patch, i.e., all $T\in\TT$ with $\overline{S}\cap\overline{T}\neq \emptyset$. The corresponding domain is denoted by $\patch(S)$.
Higher-order element patches are defined recursively by 
\begin{align*}
  \patch^{(1)}(S) = \patch(S), \quad \patch^{(n+1)}(S) = \patch(\Patch^{(n)}(S)), \quad n\in\N,
\end{align*}
where $\Patch^{(n)}(S)$ denotes the domain corresponding to $\patch^{(n)}(S)$, see Figure~\ref{fig:patch1} for a visualization.
In particular, for interior vertex $z\in \VV_0$, we simply define  $\patch_{z} := \patch(\{z\})$ and $\Omega_{z} := \Patch(\{z\})$. In Section~\ref{sec:def:general} we define $\patch_z$ for $z\in\VV\setminus\VV_0$.

Assumption~\ref{ass:vertices} implies that there exists $R=R(\TT)\in\N$ such that:
\begin{itemize}
  \item For each $v\in\VV\setminus\VV_0$ there exists $z_v\in \VV_0$ and $r_v\leq R$ with
    \begin{align*}
      v \in \partial{\Patch^{(r_v)}(\{z_v\})}.
    \end{align*}
  \item For each $E\in\EE$ there exists $z_E\in\VV_0$ and $r_E\leq R$ with
    \begin{align*}
      E\subseteq \begin{cases}
        \Patch^{(r_E)}(\{z_E\}) & \text{if } E\in \EE_0, \\
        \partial \Patch^{(r_E)}(\{z_E\}) & \text{if } E\in\EE\setminus\EE_0.
      \end{cases}
    \end{align*}
  \item For each $T\in\TT$ there exists $z_T\in \VV_0$ and $r_T\leq R$ with
    \begin{align*}
      T\subseteq \Patch^{(r_T)}(\{z_T\}).
    \end{align*}
\end{itemize}
For practical reasons, one chooses $z_\bullet$ so that $r_\bullet$ is minimal.
If $v\in(\VV\setminus\VV_0)\cap \VV_T$ and $T$ has in interior vertex $z$, then one can choose $z_v=z$, $r_v=1$. 
If $E\in\EE\cap \EE_T$ and either $E$ or $T$ has an interior vertex $z$, then one can choose $z_E = z$, $r_E = 1$.
If $T\in\TT$ has an interior vertex $z$ then one can choose $z_T = z$ and $r_T = 1$. We conclude that in most scenarios, $r_\bullet = 1$; thus, $R= 1$ holds.

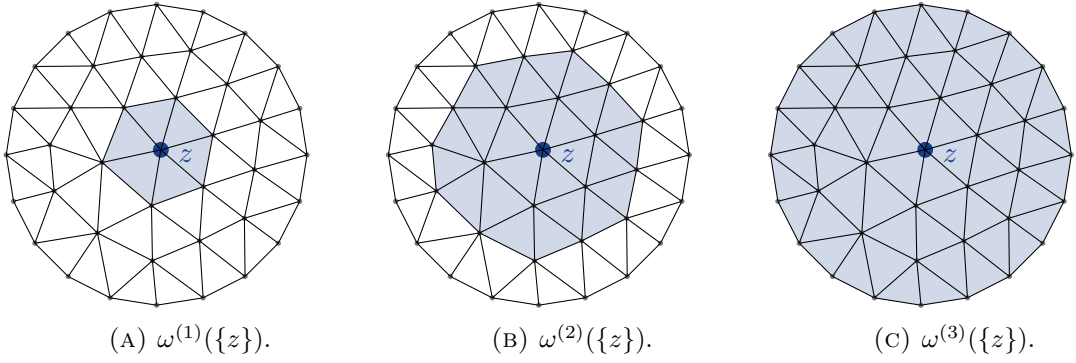
\begin{figure}
\begin{subfigure}{0.3\textwidth}
    \input{FigurePatch1}
    \caption{$\patch^{(1)}(\{z\})$.}
    \label{fig:patch1}
\end{subfigure}
\begin{subfigure}{0.3\textwidth}
\input{FigurePatch2}
 \caption{$\patch^{(2)}(\{z\})$.}
 \label{fig:patch2}
\end{subfigure}
\begin{subfigure}{0.3\textwidth}
    \input{FigurePatch3}
    \caption{$\patch^{(3)}(\{z\})$.}
    \label{fig:patch3}
\end{subfigure}
\caption{Visualization of patches.}
  \label{fig:patches}
\end{figure}

\subsection{Polynomial basis}
Let $S\subset \R^d$ denote a domain with positive measure. We denote by $P^p(S)$ the space of polynomials with domain $S$ and degree $\leq p\in\N_0$.
Let
\begin{align*}
  \Tref = \begin{cases}
    (0,1), & \text{if } d=1, \\
    \set{(x,y)\in\R^2}{x>0,y>0,x+y<1}, & \text{if } d=2,
  \end{cases}
\end{align*}
denote the reference simplex with vertices $\widehat\VV$, edges $\widehat\EE$, where $\widehat\EE=\emptyset$ for $d=1$, and consider a set of basis functions for the space $P^p(\Tref)$, $(p>0)$, associated to the vertices, edges (if $d=2$) and interior degrees of freedom, with the following properties,
\begin{subequations}\label{eq:basisTref}
\begin{align}
  & \eta^{(\Tref)}_{\zref}, \, \text{for }\zref\in\widehat\VV, \text{are the barycentric coordinates functions } (\eta^{(\Tref)}_{\zref}(z) = \delta_{z,\zref}), \\ 
  &  \eta^{(\Tref)}_{\Eref,j}, \,\text{for } j=1,\ldots,N_p, \, \Eref\in\widehat\EE, \,\text{with } \eta^{(\Tref)}_{\Eref,j}|_{E} = 0 \text{ for } E\in\widehat\EE\setminus\{\Eref\}, \\
  & \eta^{(\Tref)}_{\Tref,k}, \,\text{for } k=1,\ldots,M_p, \,\text{with }\eta^{(\Tref)}_{\Tref,k}|_{\partial\Tref} = 0. 
\end{align}
\end{subequations}
Here and throughout this work, $\delta_{\cdot,\cdot}$ denotes the Kronecker-$\delta$.

For $p=0$ we use the constant function $\chi_{\Tref} \equiv 1$ as basis for $P^0(\Tref)$.
Note that $\norm{\eta^{(\Tref)}_{\zref}}{L^\infty(\Tref)}=  1$ and we also assume that
\begin{align*}
  \norm{\eta^{(\Tref)}_{\Eref,j}}{L^\infty(\Tref)} \eqsim 1, \quad \norm{\eta^{(\Tref)}_{\Tref,k}}{L^\infty(\Tref)}\eqsim 1.
\end{align*}
Using an affine bijection $F_T\colon \Tref\to T$, for $T\in\TT$, one defines a basis of $P^p(T)$ ($p>0$) denoted by 
\begin{align*}
  \{\eta^{(T)}_{z}:\, z\in\VV_T\}\cup \{ \eta^{(T)}_{E,j}:\, j=1,\ldots,N_p, E\in\EE_T\}\cup \{\eta^{(T)}_{T,k}:\,  k=1,\ldots,M_p\},
\end{align*}
where $N_p$ and $M_p$ correspond to the number of interior basis functions per edge and element, respectively, i.e.
\[
N_p = \begin{cases}
    0,& \text{if }d=1, \\
    p-1,& \text{if }d=2,
\end{cases}, \quad 
M_p = 
\begin{cases}
    p-1,& \text{if }d=1, \\
    (p-1)(p-2)/2,& \text{if }d=2.
\end{cases}
\]
For the construction of hierarchical basis functions on a reference triangle, see, e.g.,~\cite[Section~2.2]{HighOrderFEM}.

Here, and throughout this document, for a domain $S\subset \R^d$, $L^2(S)$ denotes the Lebesgue space of square-integrable functions whose inner product and respective norm are denoted by $\ip{\cdot}\cdot_S$ and $\norm{\cdot}S$.
Furthermore, we define $H^1(S) = \set{v\in L^2(S)}{\nabla v\in L^2(S)^d}$, $H_0^1(S) = \set{v\in H^1(S)}{v|_{\partial S} = 0}$ with dual space $H^{-1}(S) = \big(H_0^1(S)\big)'$ where duality is understood with respect to the extended $L^2(S)$ scalar product. 
We also use the notation
\begin{align*}
  \norm{D^k u}\Omega^2 = \sum_{\boldsymbol{\alpha}\in \N_0^d, \, |\boldsymbol{\alpha}|=k} \norm{D^{\boldsymbol{\alpha}}u}\Omega^2.
\end{align*}

As usual, we glue together the polynomial functions on element interfaces to obtain the space of continuous piecewise polynomials, namely
\begin{align*}
  P^p_c(\TT) = P^p(\TT)\cap H^1(\Omega),\quad\text{where }
  P^p(\TT) = \set{v\in L^2(\Omega)}{v|_T\in P^p(T), \, T\in\TT}.
\end{align*}
We also define the space $P^p_{c,0}(\TT) = P^p(\TT)\cap H_0^1(\Omega).$

Defining  the index sets
\begin{align*}
  \II_{p} &:= \{z:\,z\in\VV \}  \cup \{(E,j):\, j=1,\dots,N_{p}, \,E\in\EE\} \cup \{(T,k):\, k=1,\dots,M_{p},T\in\TT\}, \\
  \II_{p,0} &:= \{z:\,z\in\VV_0 \}  \cup \{(E,j):\, j=1,\dots,N_{p}, \,E\in\EE_0\} \cup \{(T,k):\, k=1,\dots,M_{p},T\in\TT\},  \\
  \II_{p,T} &:= \{z:\,z\in\VV_T \}  \cup \{(E,j):\, j=1,\dots,N_{p}, \,E\in\EE_T\} \cup \{(T,k):\, k=1,\dots,M_{p}\},
\end{align*}
it follows that the set of functions $\{\eta_{\star}:\,\star\in \II_p\}$ is a a basis of $P^p_c(\TT)$
and the set $\set{\eta_\star}{\star\in\II_{p,0}}$ is a basis of $P^p_{c,0}(\TT)$.

We end this section defining a canonical extension for polyomials, $\Lambda_{D,D'}^{(p)}\colon P^p(D)\to P^p(D')$, where $D, D'$ satisfy $D\subseteq D'$, i.e., $(\Lambda_{D,D'}^{(p)}q)|_D = q$ for any $q\in P^p(D)$.  When no confusion is possible we simply write $q$ instead of $\Lambda_{D,D'}^{(p)} q$.

\section{Quasi-interpolator based on piecewise polynomials}\label{sec:qintpz}
In this section, we define the first family of quasi-interpolation operators with piecewise polynomial weight functions of degree up to $p$.
We begin the presentation in Section~\ref{sec:qintpz:aux} with some auxiliary results, including the critical Lemma~\ref{lem:fullrank}. 
Section~\ref{sec:def} defines and analyzes the novel quasi-interpolation operators with vanishing boundary values, and the corresponding main result on properties of these operators is presented in Theorem~\ref{thm:qintpz}. 
The crucial aspect of constructing weight functions, which holds significant interest for implementation, is thoroughly discussed in Section~\ref{sec:constr}.
A particular construction for the lowest-order case is recalled from the literature in Section~\ref{sec:constr:lowest}. 
Section~\ref{sec:def:general} delves into the definition and analysis of quasi-interpolators that operate without constraints on boundary values, a feature that significantly broadens their potential applications.

\subsection{Auxiliary results}\label{sec:qintpz:aux}
From properties of polynomials and choice of basis functions we deduce the following result.
\begin{lemma}\label{lem:independence}
  Let $T,T'\in\TT$ denote two distinct elements and $T\cup T'\subset S$. Consider $q\in P^p(S)$, $p\geq 1$, with
  \begin{align*}
    q &= \sum_{z\in \VV_T} \alpha_z \Lambda^{(p)}_{T,S}(\eta_z|_T) + \sum_{E\in\EE_T}\sum_{j=1}^{N_p} \alpha_{E,j} \Lambda^{(p)}_{T,S}(\eta_{E,j}|_T) + \sum_{k=1}^{M_p} \alpha_{T,k} \Lambda^{(p)}_{T,S}(\eta_{T,k}|_T) \\
    &= \sum_{z\in \VV_{T'}} \alpha_z' \Lambda^{(p)}_{T',S}(\eta_z|_{T'}) + \sum_{E\in\EE_{T'}}\sum_{j=1}^{N_p} \alpha_{E,j}' \Lambda^{(p)}_{T',S}(\eta_{E,j}|_{T'}) + \sum_{k=1}^{M_p} \alpha_{T',k} \Lambda^{(p)}_{T',S}(\eta_{T',k}|_{T'}).
  \end{align*}

  Therefore:
  \begin{itemize}
      \item If $z\in \VV_T\cap \VV_{T'}$ then $\alpha_z = \alpha_z'$.
      \item If $E\in \EE_T\cap \EE_{T'}$ then $\alpha_{E,j} = \alpha_{E,j}'$ for $j=1,\dots,N_p$. 
  \end{itemize}
\end{lemma}
\begin{proof}
  Suppose that $z\in \VV_T\cap \VV_{T'}$. Then, $q(z) = \alpha_z = \alpha_z'$ since all other basis functions vanish at vertex $z$.

  Suppose that $E\in \EE_T\cap \EE_{T'}$. Note that $\Lambda^{(p)}_{T,S}(\eta_{E',j}|_T)|_{E}=0=\Lambda^{(p)}_{T',S}(\eta_{E',j}|_{T'})|_E$ for all edges $E'\neq E$.
  By taking the restriction of $q$ onto $E$ it follows that
  \begin{align*}
    \sum_{z\in\VV_E} \alpha_z\Lambda^{(p)}_{T,S}(\eta_z|_{T})|_E  + \sum_{j=1}^{N_p} \alpha_{E,j} \Lambda^{(p)}_{T,S}(\eta_{E,j}|_T)|_E 
    &= \sum_{z\in\VV_E}\alpha_z'\Lambda^{(p)}_{T',S}(\eta_z|_{T'})|_E + \sum_{j=1}^{N_p} \alpha_{E,j}' \Lambda^{(p)}_{T',S}(\eta_{E,j}|_{T'})|_E
  \end{align*}
  and, since $\alpha_z = \alpha_z'$ and $\Lambda^{(p)}_{T,S}(\eta_z|_T)|_E=\Lambda^{(p)}_{T',S}(\eta_z|_{T'})|_E$ for $z\in \VV_E$, we conclude
  \begin{align*}
    \sum_{j=1}^{N_p} \alpha_{E,j} \Lambda^{(p)}_{T,S}(\eta_{E,j}|_T)|_E 
    = \sum_{j=1}^{N_p} \alpha_{E,j}' \Lambda^{(p)}_{T',S}(\eta_{E,j}|_{T'})|_E.
  \end{align*}
  Furthermore, $\Lambda^{(p)}_{T,S}(\eta_{E,j}|_T)|_E = \eta_{E,j}|_E = \Lambda^{(p)}_{T',S}(\eta_{E,j}|_{T'})|_E$. Thus, we conclude $\alpha_{E,j}=\alpha_{E,j}'$ for all $j=1,\dots,N_p$.
\end{proof}

The construction of the quasi-interpolator is based on the following result. Its proof is given for $d=1$. We provide the details for the case $d=2$ in Appendix~\ref{sec:proof2d}, following the approach outlined in~\cite{KoornwinderSauter15}, but with some minor modifications.

\begin{lemma}\label{lem:fullrank}
Let $z\in\VV_0$ and $p\in\N_0$ be given. Consider the $L^2$-orthogonal projection
\begin{align*}
  \Pi_{\patch_z}^p \colon L^2(\Patch_z) \mapsto P^p(\patch_z). 
\end{align*}
Then, $\Pi_{\patch_z}^p|_{P^{p+1}(\Patch_z)}$ has full rank, i.e., $\ker(\Pi_{\patch_z}^p|_{P^{p+1}(\Patch_z)}) = \{\nulo\}$.
\end{lemma}

\begin{proof}
Suppose that $d=1$.
Let $q\in\ker(\Pi_{\patch_z}^p|_{P^{p+1}(\Patch_z)})$, thus $q\in P^{p+1}(\Patch_z)$ and 
\begin{equation*}
  \int_{T} q(x) \,v(x)\, \di x = 0, \quad \forall v\in P^{p}(T), \forall T\in \patch_{z}.
\end{equation*}
It follows that $q\in \linhull\{ \ell_{p+1}(T)\} $, for all $T\in \patch_{z}$, 
where $\ell_{p+1}(T)$ is the Legendre polynomial of degree $p+1$ on $T$. 
Then, $q$ has at least $p+1$ distinct roots in $T$, for each $T\in \patch_z$. 
Since $\#\patch_z=2$, then $q\in P^{p+1}(\Patch_z)$ has at least $2(p+1)$ distinct roots and we conclude $q=\nulo$.
\end{proof}

\subsection{Definition {with vanishing traces}}\label{sec:def}
To define the quasi-interpolator we first need to introduce the weight functions for the vertex, edges, and element interior nodes.

\begin{subequations}\label{eq:defphi}
  \noindent \textit{Vertex weight function.} For any $z\in\VV_0$, fix some $T_z\in \patch_z$. Define the extended basis function  $q_{\star,z} := \Lambda_{T_{z}, \Patch_z}^{p+1}(\eta_{\star}|_{T_z})$, for a given index  $\star \in \II_{p+1, T_{z}}$. Let  $\phi_z\in P^p(\patch_z)$ be such that
\begin{align}\begin{split}\label{eq:defphi:z}
    \ip{\phi_z}{q_{\star,z}}_{\Patch_z} &= \delta_{z,\star} \quad\forall \star\in\II_{p+1,T_z}, \\
    \norm{\phi_z}{L^\infty(\Patch_z)}&\lesssim \frac{1}{|\Patch_z|}.
\end{split}
\end{align}
\textit{Edge weight function.} For any $E\in\EE_0$, choose $z_E\in\VV_0$ and $r_E\in\N$ such that \( E\subset \Patch^{(r_E)}(\{z_E\}).\) Set $\patch_E = \patch^{(r_E)}(\{z_E\})$ with domain $\Patch_E:=\Patch^{(r_E)}(\{z_E\})$. Moreover, fix some $T_E\in\patch_{E}$ with $E\in\EE_{T_E}$. Define the extended basis function $q_{\star, E}:= \Lambda_{T_E, \Patch_E}^{(p+1)}(\eta_{\star}|_{T_E})$, for $\star \in \II_{p+1, T_E}$. Let $\phi_{E,j}\in P^p(\patch_E)$, for $j=1,\dots,N_{p+1}$, be such that 
\begin{align}\begin{split}\label{eq:defphi:Ej}
  \ip{\phi_{E,j}}{q_{\star,E}}_{\Patch_E} &= \delta_{\star,(E,j)} \quad\forall \star\in \II_{p+1,T_E}, \\
  \norm{\phi_{E,j}}{L^\infty(\Patch_E)}&\lesssim \frac{1}{|\Patch_E|}.
\end{split}
\end{align}
\textit{Element weight function.} For any $T\in\TT$, choose $z_T\in\VV_0$ and $r_T\in\N$ such that \(T\subset\Patch^{(r_T)}(\{z_T\}).\) Set  $\patch_T = \patch^{(r_T)}(\{z_T\})$ with domain $\Patch_T$. Define the extended basis function $q_{\star,T} := \Lambda_{T, \Patch_T}^{p+1}(\eta_{\star}|_T)$ for $\star\in \II_{p+1, T}$. Let $\phi_{T,k}\in P^p(\patch_{T})$, $k=1,\dots,M_{p+1}$ be such that
\begin{align}\begin{split}\label{eq:defphi:Tk}
    \ip{\phi_{T,k}}{q_{\star,T}}_{\Patch_T} &= \delta_{\star,(T,k)} \quad\forall \star\in \II_{p+1,T}, \\
    \norm{\phi_{T,k}}{L^\infty(\Patch_T)}&\lesssim \frac{1}{|\Patch_T|}.
\end{split}
\end{align}

The existence of functions is guaranteed by Lemma~\ref{lem:fullrank}. 
However, we stress that $\phi_\star$ is not uniquely determined by the relations above. 
In Section~\ref{sec:constr} we discuss one possibility to compute the weights $\phi_\star$ by solving local problems in more detail.

We extend all $\phi_\star$, $\star\in\II_{p+1,0}$, by zero onto the whole domain $\Omega$, which implies that
\begin{align}
  \phi_\star\in P^p(\TT).
\end{align}
\end{subequations}

\noindent \textit{The quasi-interpolation operator.} We are now in a position to define our quasi-interpolator. Suppose that $\phi_\star$, $\star\in\II_{p+1,0}$, satisfy~\eqref{eq:defphi}.
For any $v\in L^1(\Omega)$, define $\qintpz v\in P_{c,0}^{p+1}(\TT)$ by
\begin{align}\label{def:quasiint}\begin{split}
  \qintpz v &= \sum_{\star\in\II_{p+1,0}} \ip{v}{\phi_\star}_\Omega \eta_\star 
  \\
  &= \sum_{z\in\VV_0} \ip{v}{\phi_z}_\Omega\eta_z + \sum_{E\in\EE_0}\sum_{j=1}^{N_{p+1}}\ip{v}{\phi_{E,j}}_\Omega\eta_{E,j} 
  + \sum_{T\in\TT}\sum_{k=1}^{M_{p+1}}\ip{v}{\phi_{T,k}}_\Omega\eta_{T,k}.
\end{split}
\end{align}
Next, we state the main result of this paper.
\begin{theorem}\label{thm:qintpz}
  Operator $\qintpz\colon L^2(\Omega)\to P^{p+1}_{c,0}(\TT)$ from~\eqref{def:quasiint} is well defined and satisfies the following properties
  \begin{enumerate}[label={\upshape(\roman*)}, align=right, widest=iii]
    \item $\qintpz = \qintpz\Pip$. 
    \item $\qintpz q|_T = q|_T$ for all $q\in P^{p+1}(\Patch^{(R)}(T))$ with $q|_{\partial\Omega\cap \partial\Patch^{(R)}(T)}=0$, $T\in\TT$.
    \item $\norm{\qintpz v}{T} \lesssim \norm{v}{\Patch^{(R)}(T)}$ for $v\in L^2(\Omega)$, $T\in\TT$.
    \item $\norm{\nabla\qintpz v}{T} \lesssim \norm{\nabla v}{\Patch^{(R)}(T)}$ for $v\in H_0^1(\Omega)$, $T\in\TT$.
    \item $\norm{(1-\qintpz)v}{T} \lesssim h_T^m\norm{D^m v}{\Patch^{(R)}(T)}$ for $v\in H^m(\Omega)\cap H_0^1(\Omega)$, $T\in\TT$, $m=1,\cdots,p+2$.
  \end{enumerate}
  The constants obtained in the inequalities depend only on the shape-regularity, $R$, and $p\in\N_0$.
\end{theorem}
\begin{proof}
  We first observe that the weight functions are elements of $L^\infty(\Omega)$, and therefore $\qintpz$ is well defined for any $v\in L^1(\Omega)\supseteq L^2(\Omega)$.

  The identity \text{(i)}, $\qintpz = \qintpz\Pip$, follows from the fact that the weight functions are elements of $P^p(\TT)$, i.e., for $v\in L^2(\Omega)$ we have that
\begin{align*}
    \qintpz v &= \sum_{\star\in\II_{p+1,0}} \ip{v}{\phi_\star}_\Omega\eta_\star = \sum_{\star\in\II_{p+1,0}} \ip{\Pip v}{\phi_\star}_\Omega\eta_\star = \qintpz\Pip v.
  \end{align*}
  Next, we show \text{(iii)}, $\norm{\qintpz v}{T} \lesssim \norm{v}{S}$, where $S=\Patch^{(R)}(T)$ for some $T\in\TT$.
  Observe that $|S| \eqsim h_T^d$ with constants depending on shape-regularity and $R$.
  From that, the properties of the weight functions and standard scaling arguments, we conclude that $\norm{\phi_\star}{S}\norm{\eta_\star}{S} \lesssim 1$.
  Together with the triangle inequality and the Cauchy--Schwarz inequality we obtain 
  \begin{align*}
    \norm{\qintpz v}{T} &\leq \sum_{\star\in\II_{p+1,T}\cap\II_{p+1,0}} |\ip{v}{\phi_\star}_\Omega|\norm{\eta_\star}T \leq \sum_{\star\in\II_{p+1,T}\cap\II_{p+1,0}} \norm{v}{S} \norm{\phi_\star}S \norm{\eta_\star}S \lesssim \norm{v}S.
  \end{align*}
  In the next step, we prove \text{(ii)}, that is, $\qintpz$ locally preserves polynomials of order $p+1$. 
  Note that $q$ has the representation 
  \begin{align*}
    q = \sum_{z\in\VV_T\cap\VV_0} \alpha_z \Lambda_{T,S}^{(p+1)}(\eta_z|_T) + \sum_{E\in\EE_T\cap\EE_0} \sum_{j=1}^{N_{p+1}} \alpha_{E,j}\Lambda_{T,S}^{(p+1)}(\eta_{E,j}|_T) + \sum_{k=1}^{M_{p+1}} \alpha_{T,k}\Lambda_{T,S}^{(p+1)}(\eta_{T,k}|_T).
  \end{align*}
  Let $z\in \VV_T\cap\VV_0$ and write
  \begin{align*}
    q|_{\Omega_z} = \sum_{\star\in\II_{p+1,T}\cap\II_{p+1,0}} \alpha_\star' q_{\star,z}.
  \end{align*}
  By Lemma~\ref{lem:independence} we have $\alpha_z = \alpha_z'$ and~\eqref{eq:defphi:z} implies $\ip{q}{\phi_z}_\Omega = \ip{q|_{\Omega_z}}{\phi_z}_{\Omega} = \alpha_z'=\alpha_z$.
  Let $E\in\EE_T\cap\EE_0$ be given and write 
  \begin{align*}
    q|_{\Omega_E} = \sum_{\star\in\II_{p+1,T}\cap\II_{p+1,0}} \alpha_\star' q_{\star,E}.
  \end{align*}
  Since $E\in\EE_{T}\cap\EE_{T_E}$ we have that $\alpha_{E,j}=\alpha_{E,j}'$ for $j=1,\ldots,N_{p+1}$ by Lemma~\ref{lem:independence}.
  Then, by construction~\eqref{eq:defphi} we have that $\alpha_{E,j}' = \ip{q}{\phi_{E,j}}_\Omega$. We conclude that $\alpha_{E,j} = \ip{q}{\phi_{E,j}}_\Omega$ for all $E\in\EE_0$, $j=1,\ldots,N_{p+1}$.
  Finally, given $k$ we write 
  \begin{align*}
    q|_{\Omega_T} = \sum_{\star\in\II_{p+1,T}\cap\II_{p+1,0}} \alpha_\star' q_{\star,T}
  \end{align*}
  and employing~\eqref{eq:defphi} yields $\alpha_{T,k} =\alpha_{T,k}' = \ip{q}{\phi_{T,k}}_\Omega$. 
  Overall, we have shown that 
  \begin{align*}
    \qintpz q|_T = \sum_{\star\in\II_{p+1,T}\cap\II_{p+1,0}} \ip{\phi_\star}{q}_\Omega \eta_\star|_T = \sum_{\star\in\II_{p+1,T}\cap\II_{p+1,0}} \alpha_\star \eta_\star|_T = q|_T
  \end{align*}
  for all $q\in P^{p+1}(S)$ with $q|_{\partial S\cap \partial\Omega} = 0$.

  We prove the approximation property \text{(v)} next. To that end, let $m\in\{1,\dots,p+2\}$ and $q\in P^{m-1}(S)$. Then, by the polynomial preserving property and local boundedness, we have that
  \begin{align*}
    \norm{(1-\qintpz)v}T &= \norm{(1-\qintpz)(v-q)}T \lesssim \norm{v-q}{S} \lesssim h_T^m\norm{D^m v}{S}
  \end{align*}
  where the last estimate follows from the well-known Bramble--Hilbert lemma and $h_T\eqsim \diam(S)$. 

  It remains to show property \text{(iv)} $\norm{\nabla \qintpz v}{T} \lesssim \norm{\nabla v}{S}$. Let $q\in P^0(S)$ (with $q|_{\partial S\cap \partial\Omega} = 0$ if the surface measure of $\partial S\cap \partial\Omega$ is positive) be given.
  Then with $\nabla q = 0$, the polynomial preserving property, an inverse estimate, and local boundedness we conclude
  \begin{align*}
    \norm{\nabla\qintpz v}{T} = \norm{\nabla \qintpz(v-q)}T \lesssim h_T^{-1}\norm{\qintpz(v-q)}T \lesssim h_T^{-1} \norm{v-q}{S} \lesssim \norm{\nabla v}{S}.
  \end{align*}
  The last estimate follows as for the approximation property.
  This finishes the proof.
\end{proof}

\subsection{Definition {without boundary constraints}}\label{sec:def:general}
We now extend the definition of the quasi-interpolator to cases without vanishing boundary values. As in the previous section we begin by introducing the weight functions.

\noindent 
\textit{Vertex weight function.} For any $z\in\VV$ fix some $T_z\in \patch(\{z\})$ and $v_z\in\VV_0$, $r_z\in\N$ with \(T_z\subset \Patch_z:=\Patch^{(r_z)}(\{v_z\})\). Note that for $z\in\VV_0$ we simply choose $v_z = z$, $r_z=1$ as in Section~\ref{sec:def}. We define the extended basis function considering $\patch^{(r_z)}(\{z\})=:\patch_z$, i.e. $q_{\star, z}:=\Lambda_{T_z, \Patch_z}^{(p+1)}(\eta_{\star}|_{T_z})$. Define then the weight function $\phi_{z}\in P^{p}(\patch_z)$ by \eqref{eq:defphi:z}.

\noindent
\textit{Edge weight function.}
For any $E\in\EE$ fix some $T_E\in\TT$ with $E\in\EE_{T_E}$ and choose $z_E\in\VV_0$ and $r_E\in\N$ such that 
\( T_E\subset \Patch^{(r_E)}(\{z_E\})=:\Patch_E\). The weight function $\phi_{E,j}$ is then defined as in \eqref{eq:defphi:Ej} with $\patch_E := \patch^{(r_E)}(\{z_E\})$.

\noindent
\textit{Element weight function.}
For any $T\in\TT$ choose $z_T\in\VV_0$ and $r_T\in\N$ such that
\(T\subset\Patch^{(r_T)}(\{z_T\})=:\Patch_T\). The weight function $\phi_{T,k}$ is then defined as in \eqref{eq:defphi:Tk} with $\patch_T := \patch^{(r_T)}(\{z_T\})$.

The quasi-interpolator $\qintp$ maps a function  $v\in L^1(\Omega)$ to
\begin{align}\label{def:quasiintp}\begin{split}
    \qintp v &= \sum_{\star\in\II_{p+1}} \ip{v}{\phi_\star}_\Omega\eta_\star 
  \\ 
  &=\sum_{z\in\VV} \ip{v}{\phi_z}_\Omega\eta_z + \sum_{E\in\EE}\sum_{j=1}^{N_{p+1}}\ip{v}{\phi_{E,j}}_\Omega\eta_{E,j} 
  + \sum_{T\in\TT}\sum_{k=1}^{M_{p+1}}\ip{v}{\phi_{T,k}}_\Omega\eta_{T,k}.
\end{split}
\end{align}
The proof of the next result follows along the lines of the proof of Theorem~\ref{thm:qintpz} with obvious modifications. Therefore, we omit details.
\begin{theorem}\label{thm:qintp}
  Operator $\qintp\colon L^2(\Omega)\to P^{p+1}_{c}(\TT)$ from~\eqref{def:quasiintp} is well defined and satisfies
  \begin{enumerate}[label={\upshape(\roman*)}, align=right, widest=iii]
    \item $\qintp = \qintp\Pip$, 
    \item $\qintp q|_T = q|_T$ for all $q\in P^{p+1}(\Patch^{(R)}(T))$, $T\in\TT$,
    \item $\norm{D^\ell\qintp v}{T} \lesssim \norm{D^\ell v}{\Patch^{(R)}(T)}$ for $v\in H^\ell(\Omega)$, $T\in\TT$, $\ell=0,1$, 
    \item $\norm{(1-\qintp)v}{T} \lesssim h_T^m\norm{D^m v}{\Patch^{(R)}(T)}$ for $v\in H^m(\Omega)$, $T\in\TT$, $m=1,\cdots,p+2$.
  \end{enumerate}
  The constants obtained in the inequalities depend only on the shape-regularity, $R$, and $p\in\N_0$.
  \qed
\end{theorem}

\subsection{Construction of the {quasi-interpolator}}\label{sec:constr}
We describe a possibility to construct weight functions $\phi_\star$ that satisfy~\eqref{eq:defphi}. We focus on the operator $\qintpz$.
With the notations from the previous sections we define $\phi_z\in P^p(\patch_z)$, for $z\in\VV_0$, as the solution to the constrained minimization problem
\begin{align*}
  \min \{ \norm{\phi}{\Patch_z}:\; \phi\in P^p(\patch_z),\, \ip{\phi}{q_{\star,z}}_{\Patch_z} = \delta_{z,\star} \,\,\forall \star\in\II_{p+1,T_z}\}.
\end{align*}
This is a convex optimization problem which is equivalent to the following variational formulation:
Find $\phi_z\in P^p(\patch_z)$, $\lambda\in P^{p+1}(\Patch_z)$ such that
\begin{subequations}\label{eq:constr:mixed}
\begin{alignat}{2}
    &\ip{\phi_z}{\psi}_{\Patch_z}  + \ip{\psi}{\lambda}_{\Patch_z} &\,=\,& 0,\label{eq:constr:mixed:a}\\
    &\ip{\phi_z}{q_{\star,z}}_{\Patch_z}  &\,=\,& \delta_{z,\star}, \label{eq:constr:mixed:b} 
\end{alignat}
for all $\psi\in P^p(\patch_z)$, $\star\in \II_{p+1,T_z}$.
\end{subequations}
By the classic theory on mixed methods~\cite{BoffiBrezziFortin} this system admits a unique solution if 
\begin{align*}
  \ip{\psi}{\lambda}_{\Patch_z}=0 \quad\text{for all } \psi\in P^p(\patch_z) \quad\text{implies that } \lambda=0. 
\end{align*}
This holds true by Lemma~\ref{lem:fullrank}.
Particularly, $\phi_z$ satisfies~\eqref{eq:defphi:z} by construction, i.e.~\eqref{eq:constr:mixed:b}. The estimate in the $L^\infty$ norm follows from a scaling argument.
Furthermore, we note that from~\eqref{eq:constr:mixed:a} we see that $\phi_z$ is in the range of $\Pi_{\patch_z}^p|_{P^{p+1}(\Patch_z)}$.

The construction of the other weight functions, $\phi_{E,j}$ and $\phi_{T,k}$, follows the same ideas and we only present some details. 
Let $E\in\EE_0$ and $j\in\{1,\dots,N_{p+1}\}$ be given and define $\phi_{E,j}$ by the following variational formulation:
Find $\phi_{E,j}\in P^p(\patch_E)$, $\lambda \in P^{p+1}(\Patch_E)$ such that
\begin{subequations}\label{eq:constr:mixed:phiE}
\begin{alignat}{2}
    &\ip{\phi_{E,j}}{\psi}_{\Patch_E}  + \ip{\psi}{\lambda}_{\Patch_E}    &\,=\,& 0,\\
    &\ip{\phi_{E,j}}{q_{\star,(E,j)}}_{\Patch_E} &\,=\,& \delta_{(E,j),\star}, 
\end{alignat}
for all $\psi\in P^p(\patch_E)$, $\star\in \II_{p+1,T_E}$. 
We note that the latter system admits a unique solution by Lemma~\ref{lem:fullrank} and $\phi_{E,j}$ is in the range of $\Pi_{\patch_E}^p|_{P^{p+1}(\Patch_E)}$.
\end{subequations}

Finally, let $T\in\TT$ and $k\in\{1,\dots,M_{p+1}\}$ be given.
Then, $\phi_{T,k}$ is defined as follows:
Find $\phi_{T,k}\in P^p(\patch_T)$, $\lambda \in P^{p+1}(\Patch_T)$ such that
\begin{subequations}\label{eq:constr:mixed:phiT}
\begin{alignat}{2}
    &\ip{\phi_{T,k}}{\psi}_{\Patch_T}  +  \ip{\psi}\lambda_{\Patch_T} &\,=\,& 0,\\
    &\ip{\phi_{T,k}}{q_{\star,(T,k)}}_{\Patch_T} &\,=\,& \delta_{(T,k),\star}, 
\end{alignat}
for all $\psi\in P^p(\patch_T)$, $\star\in \II_{p+1,T}$.
Again we note that the latter system admits a unique solution by Lemma~\ref{lem:fullrank} and $\phi_{T,k}$ is in the range of $\Pi_{\patch_T}^p|_{P^{p+1}(\Patch_T)}$.
\end{subequations}

\subsection{Alternative construction in the lowest-order case}\label{sec:constr:lowest}
In this section, we recall from~\cite{MixedFEMHm1} an alternative approach to construct a quasi-interpolator for the lowest-order case $p=0$.
Let $s_T$ denote the barycenter of $T\in\TT$. 
Given $z\in \VV_0$, let $(\alpha_{z,T})_{T\in\patch_z}$ be such that
\begin{align*}
  \sum_{T\in\patch_z} \alpha_{z,T} s_T = z, \quad \sum_{T\in\patch_z} \alpha_{z,T} = 1, \quad \alpha_{z,T}\geq 0 \quad (T\in\patch_z).
\end{align*}
We note that the coefficients $\alpha_{z,T}$ are not necessarily unique, see, e.g.,~\cite[Example~10]{MixedFEMHm1}, but there always exist coefficients with the aforegoing properties. We define weight functions
\begin{align*}
  \phi_z = \begin{cases} 
    \frac{\alpha_{z,T}}{|T|} & \text{if } T\in\patch_z, \\
    0 & \text{else},
  \end{cases}
\end{align*}
and consider the operator
\begin{align*}
  J_{0}v := \sum_{z\in\VV_0} \ip{v}{\phi_z}\eta_z \quad\text{for } v\in L^1(\Omega).
\end{align*}
The following result is from~\cite[Theorem~11]{MixedFEMHm1}. It also holds for $d\geq 3$.
\begin{theorem}
  Operator $\qintpz = J_{0}\colon L^2(\Omega)\to P^1_{c,0}(\TT)$ satisfies the assertions from Theorem~\ref{thm:qintpz} with $p=0$ and $R=1$.
  \qed
\end{theorem}

\begin{remark}
  Note that $J_{0}$ is positivity preserving, i.e., $J_0 v\geq 0$ if $v\geq 0$ a.e. in $\Omega$. This follows from the fact that by definition $\phi_z\geq 0$ and $\eta_z\geq 0$. 
  \qed
\end{remark}

\section{Quasi-interpolators based on piecewise constants}\label{sec:qintzpz}
This section defines quasi-interpolators based on piecewise constant weight functions. 
The main idea is to use larger patches to define the weight functions. 
Thus, we require that the meshes contain sufficiently many elements depending on the polynomial degree. The following assumption reflects this restriction.

\begin{assumption}\label{ass:vertices:constant}
  Let $p\in\N_0$. For each interior node $z\in\VV_0$ there exists a vicinity $\tpatch_{z} \subseteq \TT$, with corresponding domain $\tPatch_{z}$, $z\in \tPatch_z$, such that the $L^2$-projection
  \begin{align*}
    \Pi_{\tpatch_{z}}^0\colon L^2(\tPatch_{z}) \to P^0(\tpatch_{z})
  \end{align*}
  satisfies $\ker(\Pi_{\tpatch_{z}}^0|_{P^{p+1}(\tPatch_{z})})=\{\nulo\}$.
\end{assumption}

\subsection{Remarks on Assumption~\ref{ass:vertices:constant}}
For the one-dimensional case, we have the following result. 
\begin{lemma}
  Let $d=1$ and $p\in\N_0$. Suppose that $\#\TT\geq \dim(P^{p+1}) = p+2$.
  For each $z\in\TT$ choose $\tpatch_{z}\subset\TT$ arbitrary with $z\in \tPatch_{z}$ and $\#\tpatch_z=\dim(P^{p+1})=p+2$.
  Then, Assumption~\ref{ass:vertices:constant} is satisfied.
\end{lemma}
\begin{proof}
  From the assumption $\#\TT\geq \dim(P^{p+1}) = p+2$ it is clear that there exists $\tpatch_z$ with the desired properties. 

  Suppose that $q\in P^{p+1}(\tPatch_{z})$ with $\ip{q}{1}_T = 0$ for all $T\in\tpatch_{z}$. Then, $q$ has at least one root in the interior of each $T\in\tpatch_{z}$. Since $q\in P^{p+1}(\tPatch_{z})$ has at least $p+2$ distinct roots we conclude that $q=0$.
\end{proof}

Proof of the validity of Assumption~\ref{ass:vertices:constant}  becomes more complicated for higher-dimensional cases. Nevertheless, the specific case $d=2$ and $p=0$ can be concluded from Lemma~\ref{lem:fullrank} with $\tpatch_z=\patch_z$. For higher-order cases, we present Algorithm~\ref{alg:vicinity} to produce patches $\tpatch_{z}$ as required in Assumption~\ref{ass:vertices:constant}.

\IncMargin{1.5em}
\begin{algorithm}[h]
	\SetKwData{Left}{left}
	\SetKwData{This}{this}
	\SetKwData{Up}{up}
	\SetKwFunction{Union}{Union}
	\SetKwFunction{FindCompress}{FindCompress}
	\SetKwInOut{Input}{Input}
	\SetKwInOut{Output}{Output}
	\Input{$z\in \VV_0$, vicinity $\tpatch_z\subseteq \TT$ of $z$.}
  \Output{Vicinity $\tpatch_z$ satisfying Assumption~\ref{ass:vertices:constant}.}
	\BlankLine

  Set a basis $\phi_k\in P^{p+1}(\tPatch_z)$, $k=1,\dots,n=\dim(P^{p+1}(\tPatch_z))$\;
	Set a basis $\chi_j\in P^0(\tpatch_z)$, $j=1,\dots,m=\dim(P^0(\tpatch_z))$\;
	Assemble the matrix \(\boldsymbol{B}_{jk} = \ip{\phi_k}{\chi_j}_{\tPatch_z}, \quad j=1,\dots,m,\,k=1,\dots,n.\)\;
	\uIf{$\ker(\boldsymbol{B}) = \{0\}$}{ terminate\;}
\Else{redefine $\tpatch_z \longleftarrow \patch(\tPatch_z)$\;}
	\caption{Vicinity algorithm for Assumption~\ref{ass:vertices:constant}.}\label{alg:vicinity}
\end{algorithm}
\DecMargin{1.5em}

Assuming that the $L^2$ projection $\Pi_\TT^0\colon P^{p+1}(\Omega)\to P^0(\TT)$ has a trivial kernel, Algorithm~\ref{alg:vicinity} terminates in finite steps.

\subsection{Definition with {vanishing trace}}\label{sec:def:P0}
We use a similar notation as in Section~\ref{sec:def}: Let $p\in\N_0$ be given. For any $z\in\VV_0$, let $\tpatch_z\subseteq \TT$ denote the vicinity satisfying Assumption~\ref{ass:vertices:constant}. We proceed to define the weight functions for the quasi-interpolators based on piecewise constants.

\begin{subequations}\label{eq:defphi:P0}
\textit{Vertex weight functions.}
For any $z\in\VV_0$, fix some $T_z\in \tpatch_z$ with $z\in\VV_{T_z}$. Define the extended basis function 
\(q_{\star,z} := \Lambda^{(p+1)}_{T_z,\tPatch_{z}}(\eta_{\star}|_{T_z}) \). The weight function $\phi_{z}\in P^{0}(\tpatch_z)$ satisfies: 
\begin{align}\begin{split}\label{eq:defphi:z:P0}
    \ip{\phi_z}{q_{\star,z}}_{\tPatch_z} &= \delta_{z,\star} \quad\forall \star\in\II_{p+1,T_z}, \\
    \norm{\phi_z}{L^\infty(\tPatch_z)}&\lesssim \frac{1}{|\tPatch_z|}.
\end{split}
\end{align}

\textit{Edge weight functions.}
For any $E\in\EE_0$, fix $T_E$ with $E\in\EE_{T_E}$ and choose $z_E\in\VV_0$, $r_E\in\N$, $\tpatch_E\subseteq\TT$ with domain $\tPatch_E$ such that 
\begin{align*}
  T_E\subset \tPatch_E \quad\text{and}\quad 
  \tPatch_{z_E}\subseteq \tPatch_E\subseteq \Patch^{(r_E)}(\{z_E\}).
\end{align*}
Define the extended basis function 
\( q_{\star,E} := \Lambda^{(p+1)}_{T_z,\tPatch_{E}}(\eta_{\star}|_{T_E}) \). The edge weight function $\phi_{E,j} \in P^{0}(\tpatch_E)$ is such that
\begin{align}\begin{split}
    \ip{\phi_{E,j}}{q_{\star,E}}_{\tPatch_E} &= \delta_{\star,(E,j)} \quad\forall \star\in \II_{p+1,T_E}, \\
    \norm{\phi_{E,j}}{L^\infty(\tPatch_E)}&\lesssim \frac{1}{|\tPatch_E|}.
\end{split}
\end{align}

\textit{Element weight functions.}
For any $T\in\TT$ choose $z_T\in\VV_0$, $r_T\in\N$, and $\tpatch_T\subseteq\TT$ with domain $\tPatch_T$ such that
\begin{align*}
  T\subset \tPatch_T \quad\text{and}\quad \tPatch_{z_T}\subseteq \tPatch_T \subseteq \patch^{(r_T)}(\{z_T\}).
\end{align*}
Define the extended basis function by \(q_{\star,T} := \Lambda^{(p+1)}_{T,\tPatch_T}(\eta_{\star}|_{T}).\)
Let $\phi_{T,k}\in P^0(\tpatch_T)$, $k=1,\dots,M_{p+1}$ be such that 
\begin{align}\begin{split}
    \ip{\phi_{T,k}}{q_{\star,T}}_{\tPatch_T} &= \delta_{\star,(T,k)} \quad\forall \star\in \II_{p+1,T}, \\
    \norm{\phi_{T,k}}{L^\infty(\tPatch_T)}&\lesssim \frac{1}{|\tPatch_T|}.
\end{split}
\end{align}
Set $R_p = \max\set{r_E,r_T}{E\in\EE_0,T\in\TT}$.
The existence of functions $\phi_\star$ is guaranteed if Assumption~\ref{ass:vertices:constant} is satisfied. 
Construction of the weight functions can be done as in Section~\ref{sec:constr} with obvious modifications.

We extend all $\phi_\star$, $\star\in\II_{p+1,0}$, by zero onto the whole domain $\Omega$, which implies that
\begin{align}
  \phi_\star\in P^0(\TT).
\end{align}
\end{subequations}

We are now in a position to define our quasi-interpolator. Suppose that $\phi_\star$, $\star\in\II_{p+1,0}$, satisfy~\eqref{eq:defphi}. 
For $v\in L^1(\Omega)$ set
\begin{align}\label{def:quasiintz}\begin{split}
  \qintzpz v &= \sum_{\star\in\II_{p+1,0}} \ip{v}{\phi_\star}_\Omega \eta_\star 
  \\
  &= \sum_{z\in\VV_0} \ip{v}{\phi_z}_\Omega\eta_z + \sum_{E\in\EE_0}\sum_{j=1}^{N_{p+1}}\ip{v}{\phi_{E,j}}_\Omega\eta_{E,j} 
  + \sum_{T\in\TT}\sum_{k=1}^{M_{p+1}}\ip{v}{\phi_{T,k}}_\Omega\eta_{T,k}.
\end{split}
\end{align}
The next result is one of our main theorems. 
\begin{theorem}\label{thm:qintzpz}
  Operator $\qintzpz\colon L^2(\Omega)\to P^{p+1}_{c,0}(\TT)$ from~\eqref{def:quasiintz} is well defined and satisfies
  \begin{enumerate}[label={\upshape(\roman*)}, align=right, widest=iii]
    \item $\qintzpz = \qintzpz\Pi_\TT^0$, 
    \item $\qintzpz q|_T = q|_T$ for all $q\in P^{p+1}(\Patch^{(R_p)}(T))$ with $q|_{\partial\Omega\cap \partial\Patch^{(R_p)}(T)}=0$, $T\in\TT$,
    \item $\norm{\qintzpz v}{T} \lesssim \norm{v}{\Patch^{(R_p)}(T)}$ for $v\in L^2(\Omega)$, $T\in\TT$,
    \item $\norm{\nabla\qintzpz v}{T} \lesssim \norm{\nabla v}{\Patch^{(R_p)}(T)}$ for $v\in H_0^1(\Omega)$, $T\in\TT$, 
    \item $\norm{(1-\qintzpz)v}{T} \lesssim h_T^m\norm{D^m v}{\Patch^{(R_p)}(T)}$ for $v\in H^m(\Omega)\cap H_0^1(\Omega)$, $T\in\TT$, $m=1,\cdots,p+2$.
  \end{enumerate}
  The involved constants depend only on shape-regularity, $R_p$, and $p\in\N_0$.
\end{theorem}
\begin{proof}
  The proof follows the same lines of argumentation as given in the proof of Theorem~\ref{thm:qintpz}. Therefore, we omit further details.
\end{proof}

\subsection{Definition without {boundary constraints}}\label{sec:def:general:P0}
For the definition of operators based on piecewise constant weight functions we follow ideas from the previous sections and use a similar notation as in Section~\ref{sec:def:general}.
For any $z\in\VV\setminus \VV_0$ fix some $T_z\in \patch(\{z\})$ and $v_z\in\VV_0$, $r_z\in\N$, $\tpatch_z\subseteq\TT$ with domain $\tPatch_z$ such that
\begin{align*}
  T_z\subset \tPatch_{z} \quad\text{and}\quad \tPatch_{v_z} \subseteq \tPatch_z \subseteq \patch^{(r_z)}(\{v_z\}).
\end{align*}
Note that for $z\in\VV_0$ we choose the domain as in Section~\ref{sec:def:P0}.
For any $E\in\EE$ fix $T_E\in\TT$ with $E\in\EE_{T_E}$ and choose $z_E\in\VV_0$, $r_E\in\N$ and $\tpatch_E\subseteq\TT$ with domain $\tPatch_E$ such that 
\begin{align*}
  T_E\subset \tPatch_E \quad\text{and}\quad \tPatch_{z_E}\subseteq \tPatch_E \subseteq \patch^{(r_E)}(\{z_E\}).
\end{align*}
For any $T\in\TT$ choose $z_T\in\VV_0$, $r_T\in\N$ and $\tpatch_T\subseteq\TT$ with domain $\tPatch_T$ such that
\begin{align*}
  T\subset\tPatch_T \quad\text{and}\quad \tPatch_{z_T} \subseteq \tPatch_T \subseteq \patch^{(r_T)}(\{z_T\}).
\end{align*}
Set $\widetilde{R}_p = \max\set{r_z,r_E,r_T}{z\in\VV,E\in\EE,T\in\TT}$.
We define weight functions $\phi_\star$, $\star\in\II_{p+1}$, as in~\eqref{eq:defphi:P0}.
The quasi-interpolator $\qintzp$ is then defined for $v\in L^1(\Omega)$ by
\begin{align}\label{def:quasiintzp}\begin{split}
    \qintzp v &= \sum_{\star\in\II_{p+1}} \ip{v}{\phi_\star}_\Omega\eta_\star 
  \\ 
  &=\sum_{z\in\VV} \ip{v}{\phi_z}_\Omega\eta_z + \sum_{E\in\EE}\sum_{j=1}^{N_{p+1}}\ip{v}{\phi_{E,j}}_\Omega\eta_{E,j} 
  + \sum_{T\in\TT}\sum_{k=1}^{M_{p+1}}\ip{v}{\phi_{T,k}}_\Omega\eta_{T,k}.
\end{split}
\end{align}
The proof of the next result follows along the lines of the proof of Theorem~\ref{thm:qintpz} with obvious modifications. Therefore, we omit details.
\begin{theorem}\label{thm:qintzp}
  Operator $\qintzp\colon L^2(\Omega)\to P^{p+1}_{c}(\TT)$ from~\eqref{def:quasiintzp} is well defined and satisfies
  \begin{enumerate}[label={\upshape(\roman*)}, align=right, widest=iii]
    \item $\qintzp = \qintzp\Pi_\TT^0$, 
    \item $\qintzp q|_T = q|_T$ for all $q\in P^{p+1}(\Patch^{(\widetilde{R}_p)}(T))$, $T\in\TT$,
    \item $\norm{D^\ell\qintzp v}{T} \lesssim \norm{D^\ell v}{\Patch^{(\widetilde{R}_p)}(T)}$ for $v\in H^\ell(\Omega)$, $T\in\TT$, $\ell=0,1$, 
    \item $\norm{(1-\qintzp)v}{T} \lesssim h_T^m\norm{D^m v}{\Patch^{(\widetilde{R}_p)}(T)}$ for $v\in H^m(\Omega)$, $T\in\TT$, $m=1,\cdots,p+2$.
  \end{enumerate}
  The involved constants depend only on shape-regularity, $\widetilde{R}_p$, and $p\in\N_0$.
  \qed
\end{theorem}

\section{Applications}\label{sec:appl}
In this section, we discuss how to use the operators from the previous section to enhance the accuracy of piecewise polynomial approximations. We begin with the presentation of the main results and then compare the postprocessing methods to known procedures from the literature (Section~\ref{sec:appl:comp}).
Furthermore, we discuss the postprocessing of solutions to some DPG and HDG methods in Section~\ref{sec:appl:dpg} and~\ref{sec:appl:hdg} below.
Finally, in Section~\ref{sec:appl:proj} we show how our quasi-interpolation operators can be used to define projection operators onto $P^p(\TT)$, which are bounded in the dual space $H^{-1}(\Omega) = \big(H_0^1(\Omega)\big)'$.

\subsection{Postprocessing method}\label{sec:appl:postproc}
The next two theorems show how to use the quasi-interpolators designed in Sections~\ref{sec:qintpz} and~\ref{sec:qintzpz} as postprocessing techniques.
\begin{theorem}\label{thm:postproc}
  Let $u_\TT\in L^2(\Omega)$ with $p\in \N_0$ be given. 
  \begin{itemize}
    \item Suppose that $u\in H^{p+2}(\Omega)\cap H_0^1(\Omega)$ and $\norm{\Pip (u-u_\TT)}\Omega = \OO(h^{p+2})$. Then, $u_\TT^\star := \qintpz u_\TT$ satisfies
      \begin{align*}
        \norm{u-u_\TT^\star}{\Omega} = \OO(h^{p+2}).
      \end{align*}
    \item Suppose that $u\in H^{p+2}(\Omega)$ and $\norm{\Pip (u-u_\TT)}\Omega = \OO(h^{p+2})$. Then, $u_\TT^\star := \qintp u_\TT$ satisfies
      \begin{align*}
        \norm{u-u_\TT^\star}{\Omega} = \OO(h^{p+2}).
      \end{align*}
  \end{itemize}
\end{theorem}
\begin{proof}
  Using the triangle inequality and the properties of $\qintpz$ collected in Theorem~\ref{thm:qintpz} we find that
  \begin{align*}
    \norm{u-u_\TT^\star}{\Omega} = \norm{u-\qintpz u_\TT}\Omega &\leq \norm{u-\qintpz u}\Omega + \norm{\qintpz(u-u_\TT)}\Omega \\
    &= \norm{u-\qintpz u}\Omega + \norm{\qintpz\Pip(u-u_\TT)}\Omega
    \\
    &\lesssim h^{p+2} \norm{u}{H^{p+2}(\Omega)} + \norm{\Pip (u-u_\TT)}\Omega = \OO(h^{p+2}).
  \end{align*}
  This proves the first assertion, the second follows the same argumentation and is therefore omitted.
\end{proof}

The proof of the next theorem is similar to the last one and, therefore, omitted.
\begin{theorem}\label{thm:postproc:P0}
  Let $u_\TT\in L^2(\Omega)$ and $p\in \N_0$ be given. 
  \begin{itemize}
    \item Suppose that $u\in H^{p+2}(\Omega)\cap H_0^1(\Omega)$ and $\norm{\Pi^0(u-u_\TT)}\Omega = \OO(h^{p+2})$. Then, $u_\TT^\star := \qintzpz u_\TT$ satisfies
      \begin{align*}
        \norm{u-u_\TT^\star}{\Omega} = \OO(h^{p+2}).
      \end{align*}
    \item Suppose that $u\in H^{p+2}(\Omega)$ and $\norm{\Pi^0( u-u_\TT)}\Omega = \OO(h^{p+2})$. Then, $u_\TT^\star := \qintzp u_\TT$ satisfies
      \begin{align*}
        \norm{u-u_\TT^\star}{\Omega} = \OO(h^{p+2}).
      \end{align*}
  \end{itemize}
\end{theorem}

\subsubsection{Comparison to existing postprocessing technique in mixed FEM}\label{sec:appl:comp}
For a comparison to the postprocessing procedure from~\cite{Stenberg91} we consider a concrete example. The dual mixed FEM for the Poisson problem with homogeneous Dirichlet boundary condition reads: Find $(\ssigma_\TT,u_\TT)\in \RT^p(\TT)\times P^p(\TT)$ such that
\begin{subequations}\label{eq:mixed:standard}
  \begin{alignat}{2}
    &\ip{\ssigma_\TT}{\ttau}_\Omega + \ip{u_\TT}{\div\ttau}_\Omega &\,=\,& 0, \\
    &\ip{\div\ssigma_\TT}{v}_\Omega &\,=\,&  \ip{-f}{v}_\Omega
  \end{alignat}
\end{subequations}
for all $(\ttau,v)\in \RT^p(\TT)\times P^p(\TT)$, where $\RT^p(\TT)\subset \Hdivset\Omega$ denotes the Raviart--Thomas space of order $p\in\N_0$. 
It is known~\cite[Theorem~2.1]{Stenberg91} that if $\Omega$ is convex then
\begin{align}\label{eq:mixed:superclose}
  \norm{\Pip u - u_\TT}\Omega \lesssim 
  \begin{cases}
    h^{p+2}\norm{u}{H^{p+2}(\Omega)} & \text{if } p>0, \\
    h^{2}\norm{u}{H^{3}(\Omega)} & \text{if }p=0.
  \end{cases}
\end{align}

\begin{remark}
  The regularity assumption $u\in H^3(\Omega)$ for $p=0$ in~\eqref{eq:mixed:superclose} can be reduced to $u\in H^2(\Omega)$ by replacing the right-hand side in~\eqref{eq:mixed:standard} with a regularized forcing term~\cite{MixedFEMHm1}.\qed
\end{remark}
We consider the following postprocessing technique, see~\cite{Stenberg91}, where we replace the right-hand side $\ip{f}{v}_T + \dual{\ssigma_\TT\cdot\normal}v_{\partial T}$ by $\ip{\ssigma_\TT}{\nabla v}_T$.
Suppose that $(\ssigma_\TT,u_\TT)\in \RT^p(\TT)\times P^p(\TT)$ is the solution of~\eqref{eq:mixed:standard}. 
Find $u_\TT^{\star,\mathrm{st}} \in P^{p+1}(\TT)$ such that
\begin{subequations}\label{eq:mixed:postproc}
\begin{align}
  \ip{\nabla u_\TT^{\star,\mathrm{st}}}{\nabla v}_T &= \ip{\ssigma_\TT}{\nabla v}_T \quad\forall v\in P^{p+1}(T),\forall T\in\TT, \\
  \Pi_\TT^0 u_\TT^{\star,\mathrm{st}}&= \Pi_\TT^0 u_\TT.
\end{align}
\end{subequations}
Define also $u_\TT^\star = \qintpz u_\TT$. Supposing $\Omega$ is convex and that $u$ is sufficiently regular so that estimate~\eqref{eq:mixed:superclose} holds, we get 
\begin{align*}
  \norm{u-u_\TT^{\star,\mathrm{st}}}\Omega \eqsim \OO(h^{p+2}), \quad\text{and}\quad
  \norm{u-u_\TT^\star}\Omega \eqsim \OO(h^{p+2})
\end{align*}
by~\cite[Theorem~3.1]{Stenberg91} and Theorem~\ref{thm:postproc}, respectively. 
We note that we use a slightly modified right-hand side in~\eqref{eq:mixed:postproc} compared to~\cite[Eq.(2.16a)]{Stenberg91} but stress that the same proof ideas apply.
Comparing $u_\TT^{\star,\mathrm{st}}$ with $u_\TT^\star$ we first observe that both yield the same enhanced order of convergence. Second, the polynomial degree of both is the same with the difference that $u_\TT^{\star,\mathrm{st}}$ is a non-conforming approximation whereas $u_\TT^\star$ is a conforming one.
Finally, we stress that $u_\TT^{\star,\mathrm{st}}$ is defined via a local PDE using both solution components $u_\TT$, $\ssigma_\TT$, whereas $u_\TT^\star$ is defined using only operator $\qintpz$ and $u_\TT$.

\subsubsection{DPG for elasticity}\label{sec:appl:dpg}
The authors of~\cite{SupConvDPGelasticity} consider the following ultraweak formulation of the elasticity problem with homogeneous Dirichlet boundary conditions: Given $f\in L^2(\Omega)^d$ find $\uu = (u,\sigma,\widehat u,\widehat\sigma)\in U$ such that
\begin{align}\label{eq:dpgultraweak}
  b(\uu,\vv) = \ip{f}v_\Omega \quad\text{for all } \vv = (v,\tau,q)\in V, 
\end{align}
where
\begin{align*}
  b(\uu,\vv) &= \ip{\sigma}{\tau}_\Omega + \ip{u}{\div_\TT\tau}_\Omega + \ip{\sigma}{\nabla_\TT v}_\Omega + \ip{\sigma}q_\TT
  -\dual{\widehat u}{\tau\cdot\normal}_{\partial\TT} - \dual{\widehat\sigma}{v}_{\partial \TT}, \\
  U &= L^2(\Omega)^d \times L^2(\Omega)^{d\times d} \times H^{1/2}_{00}(\partial\TT)\times H^{-1/2}(\partial\TT), \\
  V &= H^1(\TT)^d \times \Hdivset{\TT}\cap \LL_\mathrm{sym}^2(\Omega) \times \LL_\mathrm{skew}^2(\Omega).
\end{align*}
Here, $H^1(\TT)$, $\Hdivset{\TT}$ (tensors with square-integrable row-wise divergence) denote broken Sobolev spaces, $\LL_\mathrm{sym}^2(\Omega)$ and $\LL_\mathrm{skew}^2(\Omega)$ denote symmetric and skew-symmetric square-integrable tensor functions, respectively. 
Furthermore, $H^{1/2}_{00}(\partial\TT)$ and $H^{-1/2}(\partial\TT)$ denote trace spaces of $H_0^1(\Omega)^d$ and $\Hdivset\Omega$, and
$\dual{\cdot}{\cdot}_{\partial\TT}$ denote trace dualities. 
In order to keep the presentation short we refer for further details to~\cite{SupConvDPGelasticity}. We note that here, for simplicity, we consider the compliance tensor to be the identity.

Following~\cite{SupConvDPGelasticity} we define
\begin{align*}
  P^{p+1}_{c,0}(\partial\TT) &= \set{v\in H_{00}^{1/2}(\partial\TT)}{v|_E\in P^{p+1}(E) \text{ for all } E\in\EE},\\
  P^p(\partial\TT) &= \set{v\in H^{-1/2}(\partial\TT)}{v|_E \in P^p(E) \text{ for all } E\in\EE} ,\\
  U_\TT^{p,0} &= P^{p}(\TT)^d\times P^p(\TT)^{d\times d} \times P^{p+1}_{c,0}(\partial\TT) \times P^p(\partial\TT), \\
  V_\TT^p &= P^{p+d}(\TT)^d \times P^{p+2}(\TT)^{d\times d}\cap \LL_\mathrm{sym}^2(\Omega) \times P^p(\TT)^{d\times d}\cap \LL_\mathrm{skew}^2(\Omega).
\end{align*}

The DPG method now reads: Find $\uu_\TT\in U_\TT^{p,0}$ such that
\begin{align}\label{eq:dpgelasticity}
  \norm{b(\uu-\uu_\TT,\cdot)}{(V_\TT^p)'} = \min_{\ww_\TT\in U_\TT^{p,0}} \norm{b(\uu-\ww_\TT,\cdot)}{(V_\TT^p)'}.
\end{align}
As shown in~\cite[Theorem~2.2]{SupConvDPGelasticity} the solution $\uu_\TT$ is quasi-optimal with respect to the canonical product norm on $U$.

We consider the postprocessing scheme~\cite[Eq.(5.5)]{SupConvDPGelasticity}: Let $\uu_\TT=(u_\TT,\sigma_\TT,\widehat u_\TT,\widehat\sigma_\TT)\in U_\TT^{p,0}$ denote the solution of~\eqref{eq:dpgelasticity}.
Find $u_\TT^{\star,\mathrm{st}}\in P^{p+1}(\TT)^d$
\begin{subequations}\label{eq:dpgelasticity:post}
\begin{align}
  \ip{\varepsilon(u_\TT^{\star,\mathrm{st}})}{\varepsilon(v)}_T &= \ip{\sigma_\TT}{\varepsilon(v)}_T \quad\forall v\in P^{p+1}(\TT)^d, \\
  \Pi_T^\mathrm{rm} u_\TT^{\star,\mathrm{st}}|_T &= \Pi_T^\mathrm{rm}u_\TT|_T \quad\forall T\in\TT.
\end{align}
\end{subequations}
Here, $\varepsilon(\cdot)$ denotes the symmetric gradient and $\Pi_T^\mathrm{rm}$ denotes the $L^2(T)$ orthogonal projection onto the rigid body motions which form the kernel of the symmetric gradient operator.
Let $\uu=(u,\ssigma,\widehat u,\widehat \sigma)\in U$ denote the solution of~\eqref{eq:dpgultraweak}.
Under common regularity assumptions (assuming $\Omega$ is convex in the present situation is sufficient) it is shown that 
\begin{align*}
  \norm{u-u_\TT^{\star,\mathrm{st}}}\Omega = \OO(h^{p+2}) \quad\text{if } p>0,
\end{align*}
see~\cite[Theorem~5.4]{SupConvDPGelasticity}. We emphasize that this result only holds true if $p\neq 0$, i.e., the lowest-order case is excluded. The reason is that the space of rigid body motions is spanned by more than only constant functions.
Nevertheless, assuming that $\Omega$ is convex, we have that~\cite[Theorem~5.1]{SupConvDPGelasticity}
\begin{align*}
  \norm{\Pip u-u_\TT}{\Omega} = \OO(h^{p+2}) \quad\text{if } p\geq 0.
\end{align*}
Therefore, we can apply our proposed postprocessing procedure,
\begin{align*}
  u_\TT^\star = \qintpz u_\TT,
\end{align*}
where the operator is applied componentwise, and by Theorem~\ref{thm:postproc} we conclude 
\begin{align*}
  \norm{u-u_\TT^\star}\Omega = \OO(h^{p+2}).
\end{align*}
We stress that the latter is also valid for $p=0$.

\subsubsection{Example HDG}\label{sec:appl:hdg}
In this section, we consider an HDG method for a first-order reformulation of the Poisson problem $\Delta u = -f$, $u|_{\partial\Omega} = 0$, cf.~\cite{CockburnGopalakrishnanSayas10}. Let $p\geq 1$. The whole HDG system reads: Find $(\qq_\TT,u_\TT,\widehat u_\TT) \in P^p(\TT)^d\times P^p(\TT)\times P^p(\partial\TT)$ such that
\begin{align*}
  \ip{\qq_\TT}{\rr}_\Omega - \ip{u_\TT}{\div_\TT\rr}_\Omega + \dual{\widehat u_\TT}{\rr\cdot\normal}_{\partial \TT} &= 0, \\
  -\ip{\qq_\TT}{\nabla_\TT w}_\Omega + \dual{\widehat\qq_\TT\cdot\normal}{w}_{\partial\TT} &= \ip{f}w_\Omega, \\
  \dual{\widehat u_\TT}{\mu}_{\partial\Omega} &= 0, \\
  \dual{\widehat\qq_\TT}{\mu}_{\partial\TT\setminus\partial\Omega} &= 0,
\end{align*}
for all $(\rr,w,\mu)\in P^p(\TT)^d\times P^p(\TT) \times P^p(\partial\TT)$. Here, the numerical flux is given by
\begin{align*}
  \widehat\qq_\TT\cdot\normal = \qq_\TT\cdot\normal + \tau(u_\TT-\widehat u_\TT)
\end{align*}
and $\tau\equiv \mathcal{O}(1)$ is the stabilization parameter.
Accuracy for approximations of the scalar field variable can be enhanced by employing the previously described postprocessing technique~\eqref{eq:mixed:postproc} to the solution component $u_\TT$ of the HDG method and with $\ssigma_\TT$ in~\eqref{eq:mixed:postproc} replaced by $-\qq_\TT$, see, e.g.~\cite[Section~5]{CockburnGopalakrishnanSayas10}. 
Particularly, we note that the supercloseness property $\norm{\Pi^0(u-u_\TT)}{\Omega} = \OO(h^{p+2})$ holds, assuming the same regularity assumptions as in the previous sections. We can thus use operator $\qintzpz$ to postprocess the solution, 
$u_\TT^\star := \qintzpz u_\TT$.
Then, 
\begin{align*}
  \norm{u-u_\TT^\star}\Omega =\OO(h^{p+2})
\end{align*}
by Theorem~\ref{thm:postproc:P0}.

\subsection{Projection operators in negative order Sobolev spaces}\label{sec:appl:proj}
In this section we construct projection operators $Q^p\colon H^{-1}(\Omega) \to P^p(\TT)$ and $\widetilde Q^p\colon \widetilde H^{-1}(\Omega)\to P^p(\TT)$ with $H^{-1}(\Omega)$ and $\widetilde H^{-1}(\Omega)$ denoting the dual spaces of $H_0^1(\Omega)$ and $H^1(\Omega)$, respectively.
We follow and extend some ideas from~\cite{MultilevelNorms21,MixedFEMHm1} which consider the lowest-order case only.
For an overview of other projection operators in negative order Sobolev spaces we refer to the recent work~\cite{DieningStornTscherpel23}.
The operators constructed in~\cite{DieningStornTscherpel23} map into $P_c^{p}(\TT)$, the space of continuous piecewise polynomials, whereas the operators we construct here map onto $P^p(\TT)$, the space of piecewise polynomials.
Application of projection operators in negative order spaces include multilevel decompositions~\cite{MultilevelNorms21}, interpolation on tensor meshes in space-time domains~\cite{DieningStornTscherpel23}, and regularization of load terms~\cite{MixedFEMHm1,FHK22,VeeserZanotti18}.

Given $T\in\TT$, let $\eta_{T,\mathrm{bubble}} = \prod_{z\in\VV_T} \eta_z|_T$ denote the element bubble function and set
\begin{align*}
  P^p_\mathrm{bubble}(T) = \set{\nu \eta_{T,\mathrm{bubble}}}{\nu\in P^p(T)}\subset H_0^1(T).
\end{align*}
Define for each $T\in\TT$, the dual functions $\nu_\bullet^{(T)} \in P^p_\mathrm{bubble}(T)$ by
\begin{align*}
  \ip{\nu_\bullet^{(T)}}{\eta^{(T)}_\star}_T = \delta_{\bullet,\star} \quad\text{for all } 
  \bullet,\star \in \II_{p,T}.
\end{align*}
In the following we extend the functions $\eta^{(T)}_\star$, $\nu_\bullet^{(T)}$ by zero outside of $T$.
We need the operator $B^p\colon L^2(\Omega) \to H_0^1(\Omega)$ given by
\begin{align*}
  B^p v = \sum_{T\in\TT} \sum_{\star\in \II_{p,T}} \ip{v}{\eta^{(T)}_\star}_\Omega\nu^{(T)}_\star.
\end{align*}
For each $p\in\N_0$ define the two operators $Q^p$ and $\widetilde Q^p$ by
\begin{align*}
  Q^p \phi &= \big(\qintpz + B^p(1-\qintpz) \big)'\phi = (\qintpz)'\phi + (1-\qintpz)'(B^p)'\phi, \\
  \widetilde Q^p \phi &= \big(\qintp + B^p(1-\qintp) \big)'\phi = (\qintp)'\phi + (1-\qintp)'(B^p)'\phi, 
\end{align*}
where the adjoint operators read
\begin{align*}
  (\qintpz)'\phi &= \sum_{\star \in \II_{p+1,0}} \ip{\phi}{\eta_\star}_\Omega \phi_\star, \quad 
  (\qintp)'\phi = \sum_{\star \in \II_{p+1}} \ip{\phi}{\eta_\star}_\Omega \phi_\star,
  \quad\text{and} \\
  (B^p)'\phi &= \sum_{T\in\TT} \sum_{\star\in \II_{p,T}} \ip{\phi}{\nu^{(T)}_\star}\eta^{(T)}_\star.
\end{align*}
In the next result, we present some properties of these operators. 
\begin{theorem}\label{thm:projHm1}
  Operators $Q^p\colon H^{-1}(\Omega)\to P^p(\TT)$ and $\widetilde Q^p\colon \widetilde H^{-1}(\Omega) \to P^p(\TT)$ are well defined and projections with
  \begin{align*}
    \norm{Q^p}{H^{-1}(\Omega)\to H^{-1}(\Omega)} + \norm{Q^p}{L^2(\Omega)\to L^2(\Omega)} &\lesssim 1, \\
    \norm{\widetilde Q^p}{\widetilde H^{-1}(\Omega)\to \widetilde H^{-1}(\Omega)} + \norm{\widetilde Q^p}{L^2(\Omega)\to L^2(\Omega)} &\lesssim 1.
  \end{align*}
  Furthermore, they satisfy the approximation estimates
  \begin{align*}
    \norm{(1-Q^p)\phi}{H^{-1}(\Omega)} + \norm{(1-\widetilde Q^p)\phi}{\widetilde H^{-1}(\Omega)} \lesssim \norm{h_\TT(1-\Pip)\phi}\Omega
  \end{align*}
  for all $\phi\in L^2(\Omega)$.
  The hidden constants depend on the shape-regularity of $\TT$, $R$, and $p\in\N_0$.
\end{theorem}
\begin{proof}
  We start by noting that
  \begin{align*}
    \norm{B^p v}T \lesssim \sum_{\star\in\II_{p,T}} \norm{v}T\norm{\eta_\star}T\norm{\nu_\star}T \lesssim \norm{v}T \quad\forall T\in\TT,\forall v\in L^2(\Omega).
  \end{align*}
  In particular, $\norm{B^p}{L^2(\Omega)\to L^2(\Omega)} \lesssim 1$. Set $P = \qintpz + B^p(1-\qintpz)$. Then,
  \begin{align*}
    \norm{Pv}{T} \lesssim \norm{\qintpz v}T + \norm{(1-\qintpz)v}T \lesssim \norm{v}{\Patch^{(R)}(T)}
  \end{align*}
  by Theorem~\ref{thm:qintpz}. With the same arguments, an inverse inequality, and the approximation property from Theorem~\ref{thm:qintpz} we see that
  \begin{align*}
    \norm{\nabla Pv}T \lesssim \norm{\nabla \qintpz v}T + h_T^{-1} \norm{(1-\qintpz) v}T \lesssim \norm{\nabla v}{\Patch^{(R)}(T)}.
  \end{align*}
  Summing over all $T\in\TT$, the last estimates prove $P\colon L^2(\Omega)\to L^2(\Omega)$ and $P\colon H_0^1(\Omega)\to H_0^1(\Omega)$ are bounded in the respective norms. Therefore, the adjoint $Q^p = P'$ is bounded in the dual norms. To see that $Q^p$ is a projection, observe that by construction $\ip{B^p(1-\qintpz)v}{\eta^{(T)}_\star}_T = \ip{(1-\qintpz)v}{\eta^{(T)}_\star}_T$ for all $\star\in \II_{p,T}$, $T\in\TT$, $v\in H_0^1(\Omega)$. Thus, for any $\phi\in P^p(\TT)$ we have that
  \begin{align*}
    \ip{Q^p \phi}v_\Omega = \ip{\phi}{P v}_\Omega = \ip{\phi}{\qintpz v + B^p(1-\qintpz)v}_\Omega = \ip{\phi}v_\Omega \quad\forall v\in H_0^1(\Omega).
  \end{align*}
  For the proof of the approximation property note that for $v\in H_0^1(\Omega)$ the same arguments as we have used to prove the boundedness of $P$ show that
  \begin{align*}
    \norm{(1-P)v}{\Omega} \lesssim \norm{h_\TT\nabla v}{\Omega}. 
  \end{align*}
  Consequently, for  $\phi\in L^2(\Omega)$ 
  \begin{align*}
    \norm{(1-Q^p)\phi}{H^{-1}(\Omega)} = \sup_{v\in H_0^1(\Omega)\setminus\{0\}} \frac{\ip{(1-Q^p)\phi}{v}_\Omega}{\norm{\nabla v}{\Omega}} = \sup_{v\in H_0^1(\Omega)\setminus\{0\}} \frac{\ip{\phi}{(1-P)v}_\Omega}{\norm{\nabla v}{\Omega}} \lesssim \norm{h_\TT\phi}\Omega.
  \end{align*}
  Considering $(1-Q^p)\phi = (1-Q^p)(1-\Pip)\phi$ in the last estimate proves the asserted approximation property.
  
  Finally, the estimates for $\widetilde Q^p$ follow the same lines of proof and we omit further details.
\end{proof}

\begin{remark}
  We note that local estimates of the form
  \begin{align*}
    \norm{Q^p\phi}{T} \lesssim \norm{\phi}{\Patch^{(R)}(T)}, \quad 
    \norm{Q^p\phi}{H^{-1}(T)} \lesssim \norm{\phi}{H^{-1}(\Patch^{(R)}(T))}
  \end{align*}
  for $\phi\in L^2(\Omega)$, $T\in\TT$, hold true as well, see~\cite[Theorem~8]{MultilevelNorms21} for the case $p=0$.
  \qed
\end{remark}

\begin{remark}
  Theorem~\ref{thm:projHm1} holds true when $\qintpz$ and $\qintp$ are replaced by $\qintzpz$ and $\qintzp$ in the definition of $Q^p$ and $\widetilde Q^p$, respectively.
  \qed
\end{remark}

\section{Numerical experiments}\label{sec:numeric}
In this section, we present numerical tests of some quasi-interpolation operators defined in this work. In Section~\ref{sec:numeric:convergence} we numerically verify approximation properties of operators $\qintpz$ and $\qintzpz$.
Section~\ref{sec:numeric:mixed} and~\ref{sec:numeric:hdg} deal with our proposed applications to postprocessing in mixed FEM and HDG methods, respectively.
Slopes of triangles with respect to the mesh-size paramater $h$ in the figures of this section are indicated with numbers $\alpha$ where $\OO( (\#\TT)^{-\alpha/2}) = \OO(h^\alpha)$ for uniform meshes (all elements have comparable diameters).

\subsection{Test of approximation properties}\label{sec:numeric:convergence}
Let $\Omega = (0,1)^2$ and consider the smooth function $u(x,y) = \sin(\pi x)\sin(\pi y)$ with $u|_{\partial\Omega} = 0$.
By Theorems~\ref{thm:qintpz} and~\ref{thm:qintzpz} we expect that
\begin{align*}
  \norm{u-\qintpz u}\Omega + \norm{u-\qintzpz u}\Omega = \OO(h^{p+2})
\end{align*}
which is observed in our experiments, see Figure~\ref{fig:smooth}.
Note that for $p=0$ both operators coincide. For $p=1$ we find that the absolute values of the errors are quite similar, whereas for $p=2,3$ the error curves are shifted, i.e., the absolute errors for $p=2,3$ and operator $\qintzpz$ are higher than for operator $\qintpz$, though convergence rates are of optimal order in both cases. 
This might have something to do with the fact that for the definition of $\qintzpz$ we use larger patches compared to $\qintpz$ to define the weight functions. 
Here, we choose patches of order $p+1$, e.g., $\tPatch_z = \Patch^{(p+1)}(\{z\})$, to define the piecewise constant weight functions.
Thus, the supports of the weight functions have a larger overlap, which affects constants in the proofs of our main results.
If we choose patches of order $p$ for $p=2,3$, e.g., $\tPatch_z = \Patch^{(p)}(\{z\})$, to define the weight functions, smaller errors are observed, see Figure~\ref{fig:smoothPatches}. 

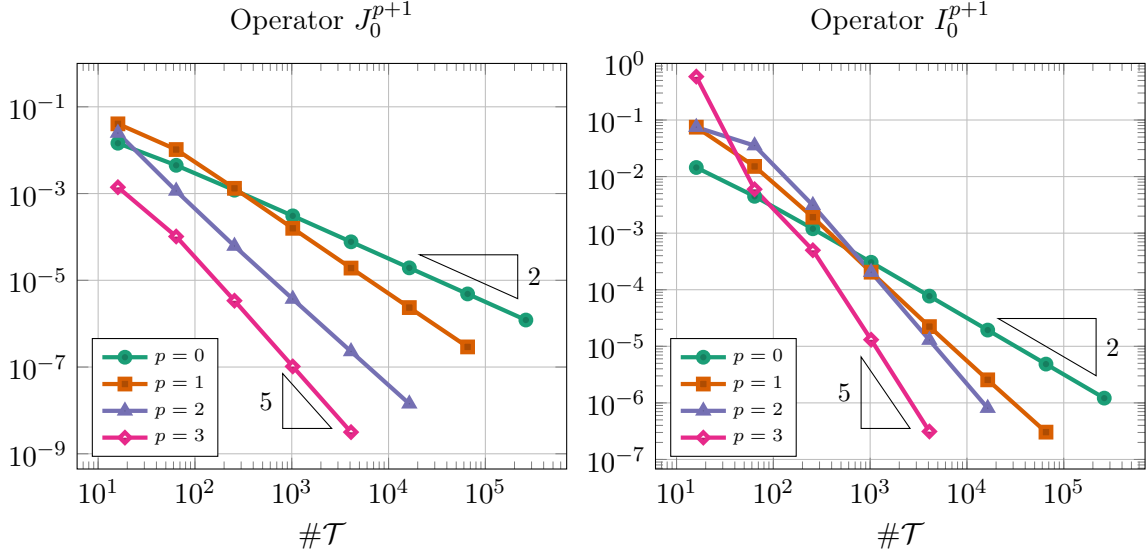
\begin{figure}
  \begin{tikzpicture}
    \begin{loglogaxis}[
        title={Operator $\qintpz$},
        width=0.49\textwidth,
        cycle list/Dark2-6,
        cycle multiindex* list={
          mark list*\nextlist
        Dark2-6\nextlist},
        every axis plot/.append style={ultra thick},
        xlabel={$\#\TT$},
        grid=major,
        legend entries={\tiny $p=0$,\tiny $p=1$,\tiny $p=2$,\tiny $p=3$},
        legend pos=south west,
        ymax=1,
      ]

      \addplot table [x=nE,y=errL2] {data/ExampleSmoothP0.dat};
      \addplot table [x=nE,y=errL2] {data/ExampleSmoothP1.dat};
      \addplot table [x=nE,y=errL2] {data/ExampleSmoothP2.dat};
      \addplot table [x=nE,y=errL2] {data/ExampleSmoothP3.dat};

      \logLogSlopeTriangle{0.9}{0.2}{0.42}{1}{black}{{\small $2$}};
      \logLogSlopeTriangleBelow{0.52}{0.1}{0.1}{5/2}{black}{{\small $5$}}
    \end{loglogaxis}
  \end{tikzpicture}
  \begin{tikzpicture}
    \begin{loglogaxis}[
        title={Operator $\qintzpz$},
        width=0.49\textwidth,
        cycle list/Dark2-6,
        cycle multiindex* list={
          mark list*\nextlist
        Dark2-6\nextlist},
        every axis plot/.append style={ultra thick},
        xlabel={$\#\TT$},
        grid=major,
        legend entries={\tiny $p=0$,\tiny $p=1$,\tiny $p=2$,\tiny $p=3$},
        legend pos=south west,
        ymax=1,
      ]

      \addplot table [x=nE,y=errL2P0] {data/ExampleSmoothP0.dat};
      \addplot table [x=nE,y=errL2P0] {data/ExampleSmoothP1.dat};
      \addplot table [x=nE,y=errL2P0] {data/ExampleSmoothP2.dat};
      \addplot table [x=nE,y=errL2P0] {data/ExampleSmoothP3.dat};

      \logLogSlopeTriangle{0.9}{0.2}{0.23}{1}{black}{{\small $2$}};
      \logLogSlopeTriangleBelow{0.52}{0.1}{0.1}{5/2}{black}{{\small $5$}}
    \end{loglogaxis}
  \end{tikzpicture}
  \caption{Errors $\norm{u-\qintpz u}\Omega$ (left) and $\norm{u-\qintzpz u}\Omega$ (right) with $u(x,y)=\sin(\pi x)\sin(\pi y)$ and domain $\Omega = (0,1)^2$.}\label{fig:smooth}
\end{figure}

\begin{figure}
  \begin{tikzpicture}
    \begin{loglogaxis}[
        title={$p=2$},
        width=0.49\textwidth,
        cycle list/Dark2-6,
        cycle multiindex* list={
          mark list*\nextlist
        Dark2-6\nextlist},
        every axis plot/.append style={ultra thick},
        xlabel={$\#\TT$},
        grid=major,
        legend entries={\tiny larger patch,\tiny smaller patch},
        legend pos=south west,
        ymax=1,
      ]

      \addplot table [x=nE,y=errL2large] {data/ExampleSmoothPatchP2.dat};
      \addplot table [x=nE,y=errL2small] {data/ExampleSmoothPatchP2.dat};

      \logLogSlopeTriangleBelow{0.85}{0.2}{0.1}{2}{black}{{\small $4$}}
    \end{loglogaxis}
  \end{tikzpicture}
  \begin{tikzpicture}
    \begin{loglogaxis}[
        title={$p=3$},
        width=0.49\textwidth,
        cycle list/Dark2-6,
        cycle multiindex* list={
          mark list*\nextlist
        Dark2-6\nextlist},
        every axis plot/.append style={ultra thick},
        xlabel={$\#\TT$},
        grid=major,
        legend entries={\tiny larger patch,\tiny smaller patch},
        legend pos=south west,
        ymax=1,
      ]

      \addplot table [x=nE,y=errL2large] {data/ExampleSmoothPatchP3.dat};
      \addplot table [x=nE,y=errL2small] {data/ExampleSmoothPatchP3.dat};

      \logLogSlopeTriangleBelow{0.85}{0.2}{0.1}{5/2}{black}{{\small $5$}}
    \end{loglogaxis}
  \end{tikzpicture}
  \caption{Errors $\norm{u-\qintzpz u}\Omega$ with $u(x,y)=\sin(\pi x)\sin(\pi y)$ and domain $\Omega = (0,1)^2$ for $p=2,3$ when using different patch sizes for the definition of the weight functions. For the larger patch a $(p+1)$-order patch is used, for the smaller one a $p$-order patch.}\label{fig:smoothPatches}
\end{figure}
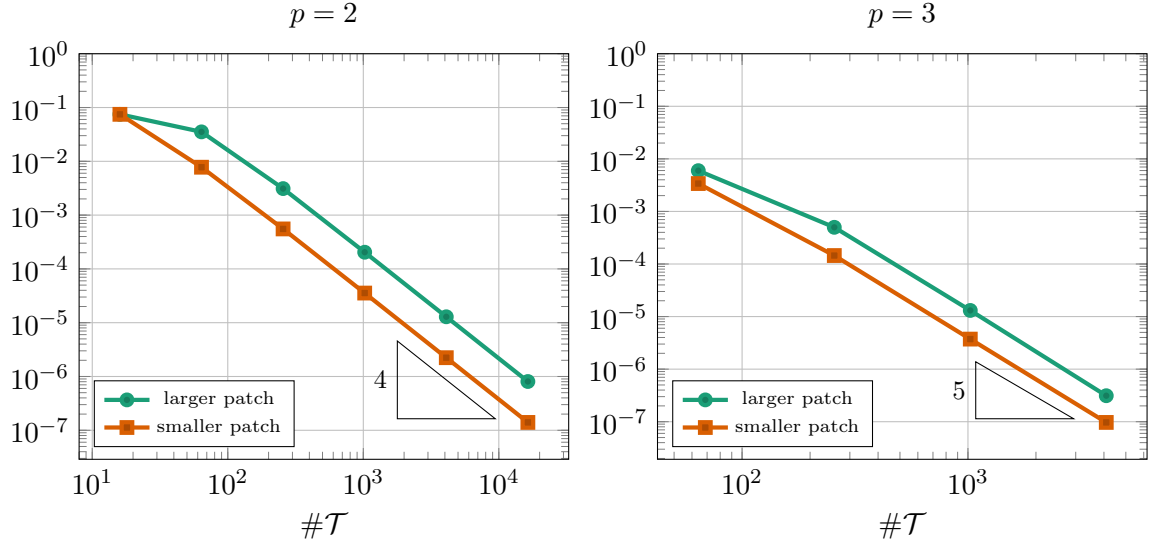

\subsection{Accuracy enhancement for mixed FEM}\label{sec:numeric:mixed}
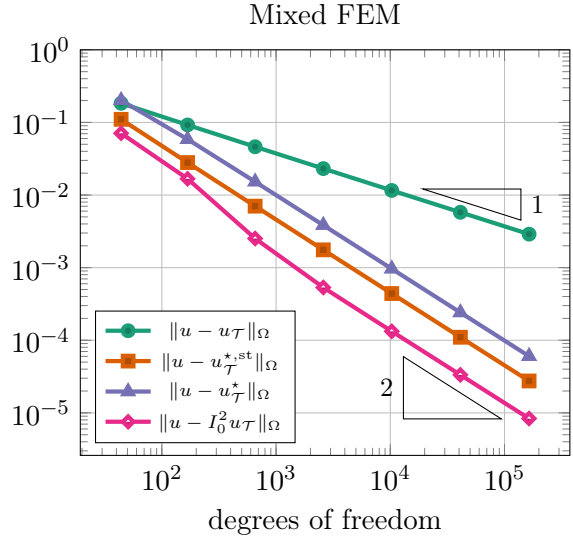
\begin{figure}
  \begin{tikzpicture}
    \begin{loglogaxis}[
        title={Mixed FEM},
        width=0.49\textwidth,
        cycle list/Dark2-6,
        cycle multiindex* list={
          mark list*\nextlist
        Dark2-6\nextlist},
        every axis plot/.append style={ultra thick},
        xlabel={degrees of freedom},
        grid=major,
        legend entries={\tiny $\norm{u-u_\TT}\Omega$,\tiny $\norm{u-u_\TT^{\star,\mathrm{st}}}\Omega$,\tiny $\norm{u-u_\TT^{\star}}\Omega$,\tiny $\norm{u-I_0^2u_\TT}\Omega$},
        legend pos=south west,
        ymax=1,
      ]

      \addplot table [x=dofMIXED,y=errUL2] {data/ExampleMixed.dat};
      \addplot table [x=dofMIXED,y=errUstarST] {data/ExampleMixed.dat};
      \addplot table [x=dofMIXED,y=errUstar] {data/ExampleMixed.dat};
      \addplot table [x=dofMIXED,y=errUstar2] {data/ExampleMixed.dat};

      \logLogSlopeTriangle{0.9}{0.2}{0.58}{1/2}{black}{{\small $1$}};
      \logLogSlopeTriangleBelow{0.86}{0.2}{0.09}{1}{black}{{\small $2$}}
    \end{loglogaxis}
  \end{tikzpicture}
  \caption{Errors of approximation and postprocessed solution for the mixed FEM described in Section~\ref{sec:numeric:mixed}.}\label{fig:mixed}
\end{figure}
We consider the mixed FEM as described in Section~\ref{sec:appl:comp} for the lowest-order case $p=0$ and $\Omega=(0,1)^2$. The manufactured solution $u(x,y) = \sin(\pi x)\sin(\pi y)$ is used. 
Using the same notation for the postprocessed solutions we find that
\begin{align*}
  \norm{u-u_\TT^\star}\Omega + \norm{u-u_\TT^{\star,\mathrm{st}}}\Omega = \OO(h^2),
\end{align*}
see Figure~\ref{fig:mixed}. We observe that the values of $\norm{u-u_\TT^{\star,\mathrm{st}}}\Omega$ are smaller than those of $\norm{u-u_\TT^\star}\Omega$. Additionally, we consider another postprocessing $I_0^2u_\TT$ which maps the piecewise constant approximation $u_\TT$ to the space $P_{c,0}^2(\TT)$. From Figure~\ref{fig:mixed} we see that this postprocessed solution converges also at $\OO(h^2)$ and the error $\norm{u-I_0^2u_\TT}\Omega$ is even smaller than $\norm{u-u_\TT^{\star,\mathrm{st}}}\Omega$.
This might be explained using the properties of $I_0^2$ from Theorem~\ref{thm:qintzpz} which yield
\begin{align*}
  \norm{u-I_0^2u_\TT}\Omega \leq \norm{u-I_0^2u}\Omega + \norm{I_0^2(u-u_\TT)}\Omega 
  \lesssim C_1 h^3 \norm{u}{H^3(\Omega)} + C_2 \norm{\Pi^0 u - u_\TT}\Omega = \OO(h^2).
\end{align*}
We see that the term containing the factor $h^3$ converges at a higher rate so that the error depends asymptotically only on $C_2 \norm{\Pi^0 u - u_\TT}\Omega$.

\subsection{Accuracy enhancement for HDG}\label{sec:numeric:hdg}
In the final numerical experiment, we consider the HDG method from Section~\ref{sec:appl:hdg}. We use again $\Omega = (0,1)^2$ and the manufactured solution $u(x,y) = \sin(\pi x)\sin(\pi y)$. We apply operator $\qintzpz$ to the HDG solution component $u_\TT$ to obtain a postprocessed solution. Figure~\ref{fig:hdg} visualizes the errors $\|u-\qintzpz u_\TT\|_\Omega$ for $p=1,2,3$.
Again we observe optimal algebraic rates, i.e., $\|u-\qintzpz u_\TT\|_\Omega = \OO(h^{p+2})$.
For the cases $p=2,3$ we use smaller patches to define $\qintzpz$ as described in Section~\ref{sec:numeric:convergence} above. 

\begin{figure}
  \begin{tikzpicture}
    \begin{loglogaxis}[
        title={Results for HDG},
        width=0.49\textwidth,
        cycle list/Dark2-6,
        cycle multiindex* list={
          mark list*\nextlist
        Dark2-6\nextlist},
        every axis plot/.append style={ultra thick},
        xlabel={$\#\TT$},
        grid=major,
        legend entries={\tiny $p=1$,\tiny $p=2$,\tiny $p=3$},
        legend pos=south west,
        ymax=1,
      ]

      \addplot table [x=nE,y=errL2small] {data/ExampleHDGP1.dat};
      \addplot table [x=nE,y=errL2small] {data/ExampleHDGP2.dat};
      \addplot table [x=nE,y=errL2small] {data/ExampleHDGP3.dat};

      \logLogSlopeTriangle{0.9}{0.2}{0.45}{3/2}{black}{{\small $3$}};
      \logLogSlopeTriangleBelow{0.85}{0.2}{0.1}{5/2}{black}{{\small $5$}}
    \end{loglogaxis}
  \end{tikzpicture}
  \caption{Errors $\|u-\qintzpz u_\TT\|_\Omega$ for $p=1,2,3$ where $u_\TT$ is the solution of the HDG method. Here, we use smaller patches for $p=2,3$ see comments and discussion in Section~\ref{sec:numeric:convergence}.}\label{fig:hdg}
\end{figure}
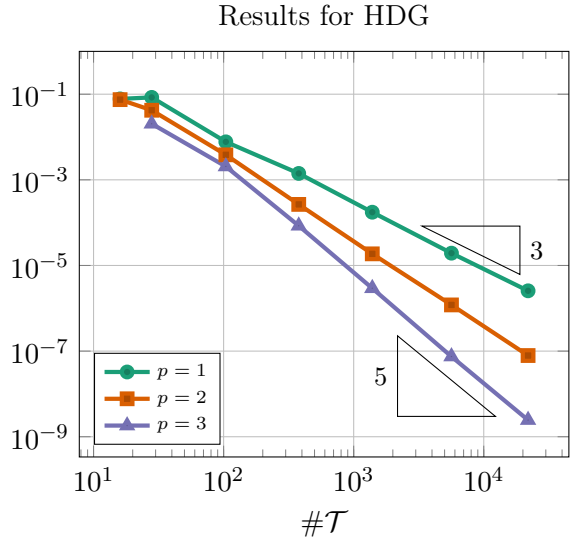

\section{Conclusion}\label{sec:conclusion}
In this work, we presented a novel family of quasi-interpolators into continuous piecewise polynomials of degree $\leq p+1$ that are based on polynomial weight functions of degree $p$. 
They are locally defined and have optimal approximation properties. The existence of the weight functions, which we verified for 1D and 2D, is critical in the construction. In the three-dimensional case, the same principal ideas in the construction of such operators apply, but the proof of the existence of corresponding weight functions remains open.
Further, we designed quasi-interpolators with piecewise constant weight functions that have similar properties; though, proof of the weight functions' existence is much more involved.

The operators can enhance the accuracy of finite element solutions without employing gradient approximations, contrary to other postprocessing techniques.
The operators can be implemented in a finite element software library and then be used for various different problems as we have exemplified in Section~\ref{sec:appl} and~\ref{sec:numeric}.

\bibliographystyle{alpha}
\bibliography{literature}


\appendix
\section{Proof of Lemma~\ref{lem:fullrank} for two dimensions}\label{sec:proof2d}
Lemma~\ref{lem:fullrank} for $d=2$ has been verified in~\cite[Theorem~1.1]{KoornwinderSauter15} using a weighted inner product on each $T\in\TT$, i.e., 
\begin{align*}
  \ip{u}v_{T,w_T} = \int_T uv\, w_T^\alpha\,\di x, \quad w_T = \prod_{z\in\VV_T}\eta_z|_T, \quad \alpha = 1.
\end{align*}
Here, we adopt the proof of~\cite[Theorem~1.1]{KoornwinderSauter15} for $\alpha = 0$, i.e., we consider the canonic inner product.
We also adopt some notation, i.e., we shall use $n=p+1$. 
We stress that the case $n=1$ has been covered in~\cite[Section~5]{KoornwinderSauter15} for finite measures invariant under affine transformations. This includes, in particular, our situation with $\alpha = 0$.
For the remainder of this section, we can thus assume $n\geq 2$.

\subsection{Main ideas of proof}
Let $z\in \VV_0$ be given. 
We consider $q\in P^{n}(\Omega_z)$ with $\ip{q}{v}_{\Omega_z} = 0$ for all $v\in P^{n-1}(\omega_z)$. We want to show that this implies $q=0$ which is equivalent to the assertion of Lemma~\ref{lem:fullrank}.
The main tool in the proof is the use of an appropriate representation of orthogonal polynomials on triangles. 
Employing an affine mapping we can assume w.l.o.g. that $\Tref \in \patch_z$.
Considering an element $T\in\patch_z$ adjacent to $\Tref$ we prove in a first step (Section~\ref{sec:appendix:adj}) that
\begin{align*}
  \ip{q}{v}_{\Tref} = 0 \quad\text{for all } v\in P^{n-1}(\Tref) \quad\text{and}\quad \ip{q}{w}_T = 0 \quad\text{for all } w\in P^{n-1}(T)
\end{align*}
implies that $q =0$ or that $q$ lives in a one-dimensional subspace for some exceptional cases.
By considering elements in $\patch_z$ and resolving the different exceptional cases we conclude that $q=0$ (Section~\ref{sec:appendix:conclusion}).

\subsection{Orthogonal polynomials}
The bivariate polynomials
\begin{align*}
  P_{n,k}(x,y) = P_{n-k}^{(0,2k+1)}(1-2x) (1-x)^k P_k^{(0,0)}(1-\tfrac{2y}{1-x})
\end{align*}
for $k=0,\dots,n$ are of degree $n\in\N_0$ and orthogonal on the reference triangle $\Tref$, i.e.,
\begin{align*}
  \ip{P_{n,k}}{P_{m,\ell}}_{\Tref} = 0 \quad (n,k)\neq (m,\ell).
\end{align*}
Here, $P_k^{(\alpha,\beta)}$ denote the Jacobi polynomials on $(-1,1)$ with $\alpha,\beta>-1$, i.e., 
\begin{align*}
  P_k^{(\alpha,\beta)}(1-2x) = \frac{(\alpha+1)_k}{k!}\sum_{j=0}^k \frac{(-n)_j(n+\alpha+\beta+1)_j}{(\alpha+1)_j j!}x^j,
\end{align*}
where $(z)_j$ for $j\in\N_{0}$ is the Pochhammer symbol given by $(z)_{j} = \prod_{k=0}^{j-1}(z+k)$ if $j>0$ and $(z)_0 = 1$.

We note that any affine transformation from an element $T\in\TT$ to $\Tref$, say, $G_T\colon T\to \Tref$, preserves the orthogonality relations, i.e., $\ip{P_{n,k}\circ G_T}{P_{m,\ell}\circ G_T}_{T} = 0$ if $(n,k)\neq (m,\ell)$. 

The orthogonal complement
\begin{align*}
  P^{n-1,n,\perp}(\Tref) = \set{q\in P^n(\Tref)}{\ip{q}r_{\Tref}=0 \, \forall r\in P^{n-1}(\Tref)}
\end{align*}
is spanned by the polynomials $P_{n,k}$, $k=0,\dots,n$. Considering the affine map that takes $(0,0)\to (0,0)$, $(1,0)$ to $(0,1)$ and $(0,1)$ to $(1,0)$ we find that the polynomials $p_{n,k}(x,y) = P_{n,k}(y,x)$, $k=0,\dots,n$ span the space $P^{n-1,n,\perp}(\Tref)$.
Setting $k=0$ we also see by similar considerations that the three bivariate polynomials
\begin{align*}
  P_n^{(0,1)}(1-2x), \quad P_n^{(0,1)}(1-2y), \quad P_n^{(0,1)}(2(x+y)-1)
\end{align*}
are elements of $P^{n-1,n,\perp}(\Tref)$, see also~\cite[Eq.(2.4) and (2.5)]{KoornwinderSauter15}.

We use the series representation
\begin{align}\label{eq:ortpol:ser}
\begin{split}
  p_{n,k}(x,y) &= \sum_{m=0}^{n-k} \frac{(k-n)_m(n+k+2)_m}{(m!)^2} y^m \sum_{j=0}^k \frac{(-k)_j(k+1)_j}{(j!)^2} x^j(1-y)^{k-j}.
\end{split}
\end{align}

\subsection{Two adjacent triangles}\label{sec:appendix:adj}
Let $T = \conv\{(1,0),(0,0),(-\tfrac{c}{d-c},\tfrac1{d-c})\}$ be an adjacent triangle to $\Tref$ with $d-c<0$ (otherwise $|\Tref\cap  T|>0$).
Take $q$ with $q|_{\Tref}\in P^{n-1,n,\perp}(\Tref)$ and $q|_T\in P^{n-1,n,\perp}(T)$. Then, 
\begin{align*}
  q|_{\Tref}(x,y) = \sum_{k=0}^n \alpha_k p_{n,k}(x,y) \quad\text{and}\quad
  q|_{T}(x,y) = \sum_{k=0}^n \beta_k p_{n,k}(x+cy,(d-c)y).
\end{align*}
Therefore, 
\begin{align}\label{eq:twotri:a}
  \sum_{k=0}^n (\alpha_k p_{n,k}(x,y)-\beta_k p_{n,k}(x+cy,(d-c)y)) = 0 \quad\text{for all } (x,y)\in\R^2.
\end{align}
Comparing the coefficients of the monomials $(x,y)\mapsto x^j$ for $j=0,\dots,n$ gives $\alpha_j=\beta_j$ for $j=0,\dots,n$.
Following along the lines of~\cite[Proof of Theorem~3.1]{KoornwinderSauter15}, particularly,~\cite[Eq.(3.9) and Eq.(14)]{KoornwinderSauter15} we find that the coefficients of the monomials $(x,y)\mapsto x^ry^m$ in~\eqref{eq:twotri:a} are given by the equations
\begin{align}\label{eq:twotri:b}
  \frac{1}{r!}\sum_{k=r}^n\alpha_kf_{k,r,m}(c,d) = 0 
\end{align}
for $r=0,\dots,n-1$, $m=1,\dots,n$, $r+m\leq n$ and 
\begin{align*}
  f_{k,r,m}(c,d) &= \sum_{i=\max(0,m-k+r)}^{\min(m,n-k)} \frac{(-n+k)_i(n+k+2)_i}{(i!)^2} \frac{(-k)_r(k+1)_r(r-k)_{m-i}}{r!(m-i)!}(1-(d-c)^m)
  \\
  &\quad - \sum_{i=\max(0,m-k+r)}^{\min(m-1,n-k)} \frac{(-n+k)_i(n+k+2)_i}{(i!)^2}
  \\&\qquad \times
  \sum_{j=r+1}^{\min(k,m+r-i)} \frac{(-k)_j(k+1)_j(j-k)_{m+r-i-j}}{j!(j-r)!(m+r-i-j)!}c^{j-r}(d-c)^{m+r-j}.
\end{align*}
Considering $r=n-j$, $j=1,\dots,n$, $m=1$ we rewrite~\eqref{eq:twotri:b} to see 
\begin{align*}
  \alpha_{n-j} f_{n-j,n-j,1}(c,d) = -\sum_{k=n-j+1}^n \alpha_{k} f_{k,n-j,1}(c,d),
\end{align*}
where
\begin{align*}
  f_{n-j,n-j,1}(c,d) = \frac{(-j)(2n-j+1)(-n+j)_{n-j}(n-j+1)_{n-j}}{(n-j)!}(1-(d-c)).
\end{align*}
Note that by our assumption we have that $d-c\neq 1$. Then, $f_{n-j,n-j,1}\neq 0$ and $\alpha_{n-j}$ is determined by $\alpha_{n-j+1}$, \dots, $\alpha_n$. Therefore, the solution space has dimension at most $1$.

We now consider $(r,m) = (n-1,1)$, $(n-2,1)$, $(n-2,2)$ in~\eqref{eq:twotri:b} which gives the $3\times 3$ system
\begin{align}\label{eq:sys3t3}
\begin{split}
  \alpha_n f_{n,n-1,1} + \alpha_{n-1}f_{n-1,n-1,1} &= 0, \\
  \alpha_n f_{n,n-2,1} + \alpha_{n-1}f_{n-1,n-2,1} + \alpha_{n-2}f_{n-2,n-2,1} &= 0, \\
  \alpha_n f_{n,n-2,2} + \alpha_{n-1}f_{n-1,n-2,2} + \alpha_{n-2} f_{n-2,n-2,2} &= 0.
\end{split}
\end{align}
The determinant of the system is
\begin{align*}
  C_n c(d-1)(1-(d-c))(1+(d-c))
\end{align*}
with $C_n\neq 0$ for all $n\geq 2$.
It vanishes only if $d-c=-1$, $c=0$, $d=1$. Consequently, if $c\neq 0$, $d\neq 1$ and $d-c\neq -1$ then, $q=0$. 
It is only necessary to study the following exceptional cases where the dimension of the solution space is possibly one.

\noindent
\textbf{Case $c=0$:}
By our aforegoing considerations we have that $P_n^{(0,1)}(1-2x)$
is an element of $P^{n-1,n,\perp}(\Tref)$ and  $P_n^{(0,1)}(1-2(x+cy))$ is an element of $P^{n-1,n,\perp}(T)$. 
If $c=0$ we conclude that
\begin{align*}
  q\in \linhull\{P_n^{(0,1)}(1-2x)\}
\end{align*}
Since the dimension space is at most $1$ this finishes the case $c=0$.

\noindent
\textbf{Case $d=1$:}
Arguing similar as in case $c=0$ we find for $d=1$ that
\begin{align*}
  q \in \linhull\{P_n^{(0,1)}(2(x+y)-1)\}
\end{align*}

\noindent
\textbf{Case $n=2$, $d=c-1$:}
One verifies that 
\begin{align*}
  q \in \linhull\{q_2\}
\end{align*}
with $q_2 = 40p_{2,2} + 24(c-1)p_{2,1} + (3c^2-6c+8)p_{2,0}$.

\noindent
\textbf{Case $n>2$, $d-c=-1$:}
Note that $(1-(d-c)^2) = 0$. Taking the first two equations of~\eqref{eq:sys3t3} we may express $\alpha_{n-1}$ and $\alpha_{n-2}$ in terms of $\alpha_n$. 
Replacing $\alpha_{n-1}$ and $\alpha_{n-2}$ in~\eqref{eq:twotri:b} for $(r,m) = (n-3,1), (n-3,3)$ we obtain a two times two homogeneous system in the unknowns $\alpha_n$ and $\alpha_{n-3}$. Its determinant vanishes if $c=0,1,2$.
Thus, if $c\neq 0,1,2$ we conclude that $q=0$. The case $c=2$ implies that $d = 1$ which, as the case $c=0$, has been handled above. 
We may thus restrict to the case $c=1$ which implies that $d=0$.
To keep the presentation short, we suppose that for $(c,d)=(1,0)$, the solution space is one-dimensional (otherwise, it would be trivial). 
We get the representation
\begin{align*}
  q = \alpha\widetilde p_n := \alpha\left(p_{n,n} + \sum_{j=1}^{n} \alpha_{n,j} p_{n,n-j} \right).
\end{align*}
Note that the coefficient in front of $p_{n,n}$ does not vanish. Otherwise, $\alpha_n = 0$ which implies that all other coefficients vanish by the previous considerations.
We do not need a more explicit representation of $q$ for our later study but note that the coefficient of the monomial $x^n$ does not vanish. 

\begin{remark}
  It is possible to derive a more explicit representation for the case $n>2$, $(c,d) = (1,0)$. Define
  \begin{align*}
    \widetilde p(x,y) = \sum_{j=0}^n P_n^{(0,0)}(1-2x)P_{n-j}^{(0,0)}(1-2y).
  \end{align*}
  Let $F_T\colon \Tref \to T = \conv\{(1,0),(0,0),(1,-1)\}$ denote the affine mapping that maps $(0,0)$ to $(0,0)$, $(1,0)$ to $(1,-1)$ and $(1,1)$ to $(1,0)$. Then, $p = \widetilde p\circ F_T^{-1}$ satisfies $p|_{\Tref} \in P^{n-1,n,\perp}(\Tref)$, $p|_T \in P^{n-1,n,\perp}(T)$.\qed
\end{remark}

To close this section we consider the element 
\begin{align*}
  T'& = \conv\{(0,0),(0,1),(1/(d'-c'),-c'/(d'-c'))\}
\end{align*}
with $d'-c'<0$. Following the arguments from above (interchanging the role of the $x$ and $y$ variables) we obtain the following results. 
Suppose that $q$ satisfy $q|_{\Tref}\in  P^{n-1,n,\perp}(\Tref)$ and $q|_{T'} \in P^{n-1,n,\perp}(T')$. 
Then, $q=0$ except in the cases below: 

\noindent
\textbf{Case $c'=0$:}
We obtain that 
\begin{align*}
  q\in \linhull\{P_n^{(0,1)}(1-2y)\}.
\end{align*}

\noindent
\textbf{Case $d'=1$:}
Arguing as in case $d=1$ we find that
\begin{align*}
  q \in \linhull\{P_n^{(0,1)}(2(x+y)-1)\}
\end{align*}

\noindent
\textbf{Case $n=2$, $d'=c'-1$:}
One verifies that 
\begin{align*}
  q \in \linhull\{q_2'\}
\end{align*}
with $q_2'(x,y) = 40p_{2,2}(y,x) + 24(c-1)p_{2,1}(y,x) + (3c^2-6c+8)p_{2,0}(y,x)$.

\noindent
\textbf{Case $n>2$, $d'-c'=-1$:}
In the final case, we obtain 
\begin{align*}
  q(x,y) = \alpha \widetilde p_n(y,x).
\end{align*}
\subsection{Intersection of orthogonal polynomials on a patch}\label{sec:appendix:conclusion}
In this section, we consider a patch $\patch_z$ with $z\in\VV_0$.
W.l.o.g. we may assume that $\Tref\in\patch_z$.
Moreover, let 
\begin{align*} 
  T &= \conv\{(1,0),(0,0),(-c/(d-c),1/(d-c))\}, \\
  T'& = \conv\{(0,0),(0,1),(1/(d'-c'),-c'/(d'-c'))\},
\end{align*}
with $d-c<0$, $d'-c'<0$.
We distinguish between various cases according to the previous section. 

Take $q\in \ker(\Pi_{\patch_z}^{n-1}|_{P^n(\Patch_z)})$. From the previous section, we already know that $q$ is either trivial or lives in a one-dimensional space depending on specific values of $c,d$, and $c',d'$. 
We only need to study the following cases.

\noindent
\textbf{Case $c=0$.}
We have $q \in \linhull\{P_n^{(0,1)}(1-2x)\}$. 
Note that all solutions $P_n^{(0,1)}(1-2x)=0$ satisfy $x\in(0,1)$. This implies that $P_n^{(0,1)}(1-2x)|_{T'}$ does not change sign, which means that
\begin{align*}
  \int_{T'} P_n^{(0,1)}(1-2x) \,\di(x,y)\neq 0.
\end{align*}
We conclude that $P_n^{(0,1)}(1-2x)|_{T'}\notin P^{n-1,n,\perp}(T')$ and therefore $q=0$.

\noindent
\textbf{Case $c'=0$.}
We have $q\in \linhull\{P_n^{(0,1)}(1-2y)\}$. Arguing as in the case $c=0$ we conclude $q=0$.

\noindent
\textbf{Case $d-c=-1$ and $n=2$ with $d'=1$.}
Here we have for some $\alpha,\beta\in\R$ that
\begin{align*}
  q|_T(x,y) &= \alpha(40p_{2,2}(x,y)+24(c-1)p_{2,1}(x,y) + (3c^2-6c+8)p_{2,0}(x,y)), \\
  q|_{T'}(x,y) &= \beta P_2^{(0,1)}(2(x+y)-1).
\end{align*}
By comparing the coefficients of the two polynomials we find after a short computation that $\alpha=0=\beta$, or, equivalently, $q=0$.

\noindent
\textbf{Case $d'-c'=-1$ and $n=2$ with $d=1$.}
Arguing as in the case $d-c=-1$, $d'=1$ (and $n=2$) we conclude $q=0$.

\noindent
\textbf{Case $d=1$, $(c',d')=(1,0)$ and $n>2$.}
Recall that $q|_T \in \linhull\{P_n^{(0,1)}(2(x+y)-1)\}$. 
We claim that $P_n^{(0,1)}(2y-1) \in P^{n-1,n,\perp}(\widetilde T')$ with $\widetilde T' = \conv\{(-1,0),(0,1),(-1,1)\}$.
To see this, note that $P_n^{(0,1)}(1-2y)\in P^{n-1,n,\perp}(\Tref)$ and consider the affine transformation that maps $(0,0)$ to $(0,1)$, $(1,0)$ to $(-1,1)$ and $(0,1)$ to $(-1,0)$. We have that
\begin{align*}
  \int_{T'} P_n^{(0,1)}(2(x+y)-1) \,\di(x,y) &= \int_{\widetilde T} P_n^{(0,1)}(2y-1)\,\di(x,y)
  \\
  &= \int_{(-1,0)\times (0,1)} P_n^{(0,1)}(2y-1)\,\di(x,y) = \int_0^1 P_n^{(0,1)}(2y-1)\,\di y  \neq 0, 
\end{align*}
where $\widetilde T = \conv\{(-1,0),(0,0),(0,1)\}$.
The last identity follows the properties of Jacobi polynomials.
We conclude that $q$ must vanish.

\noindent
\textbf{Case $(c,d)=(1,0)$ and $n>2$ with $d'=1$.}
We argue as in the previous case with $d=1$, $(c',d')=(1,0)$ and conclude that $q=0$.

\noindent
\textbf{Case $d-c=-1$ and $n=2$ with $d'=c'-1$.}
Here we have for some $\alpha,\beta\in\R$ that
\begin{align*}
  q|_T(x,y) &= \alpha(40p_{2,2}(x,y)+24(c-1)p_{2,1}(x,y) + (3c^2-6c+8)p_{2,0}(x,y)), \\
  q|_{T'}(x,y) &= \beta(40p_{2,2}(y,x)+24(c'-1)p_{2,1}(y,x) + (3c'^2-6c'+8)p_{2,0}(y,x)).
\end{align*}
After a short computation, we find that $\alpha=0=\beta$, or, equivalently, $q=0$.

\noindent
\textbf{Case $(c,d)=(1,0)$ and $n>2$ with $(c',d')=(1,0)$.}
First consider the two triangles $K_1 = \conv\{(0,0),(1,0),(1,1)\}$, $K_2 = \conv\{(0,0),(1,1),(0,1)\}$ which define a decomposition of the square domain $(0,1)^2$. Then, consider a polynomial $\widetilde q$ with $\widetilde q|_{K_1}\in P^{n-1,n,\perp}(K_1)$ and $\widetilde q|_{K_2}\in P^{n-1,n,\perp}(K_2)$.
By symmetry considerations we must have $\widetilde q(x,y) = \widetilde q(y,x)$ for all $(x,y)\in (0,1)^2$.
Moreover, the affine map $(x,y)\mapsto (x,-x+y)$ sends the domain $K_1\cup K_2$ to the domain $\Tref\cup T$ and $q_1(x,y):=\widetilde q(x,x+y)$ satisfies $q_1|_{\Tref}\in P^{n-1,n,\perp}(\Tref)$ and $q_1|_{T}\in P^{n,n-1,\perp}(T)$. By the previous section, we thus have that
\begin{align*}
  q_1 \in \linhull\{\widetilde p_n\}.
\end{align*}
In particular, the coefficient of the monomial $x^n$ in $\widetilde p_n$ is non-vanishing (see previous section). 
Therefore, the coefficient $\alpha$ of the monomial $x^n$ in $\widetilde q(x,y)$ is non-vanishing. By symmetry, we have that $\alpha$ is the coefficient of the monomial $y^n$ in $\widetilde q(x,y)$.
Then, $2\alpha$ is the coefficient of the monomial $x^n$ in $q_1$ and $\alpha$ is the coefficient of the monomial $y^n$ in $q_1$. 

Consider a second affine mapping $(x,y)\mapsto (x-y,y)$. This maps the domain $K_1\cup K_2$ to the domain $T'\cup \Tref$ and $q_2(x,y) := \widetilde q(x+y,y)$ satisfies $q_2|_{\Tref}\in P^{n-1,n,\perp}(\Tref)$ and $q_2|_{T'}\in P^{n,n-1,\perp}(T')$.
Moreover, $\alpha$ is the coefficient of $x^n$ in $q_2$ and $2\alpha$ is the coefficient of $y^n$ in $q_2$. 

We conclude that $q|_T \in \linhull\{q_2\}$ and $q|_{T'} \in \linhull\{q_2\}$. By comparing the coefficients of $x^n$ and $y^n$, it is easy to see that $q=0$, which finishes the proof of Lemma~\ref{lem:fullrank}.

\noindent
\textbf{Case $d=1$, $d'=1$.}
In the final case we have that $(-c/(1-c),1/(1-c))$ and $(1/(1-c'),-c'/(1-c'))$ are points on the line connecting $(1,0)$ and $(0,1)$.
Clearly, not all triangles in the patch have points on this very same line. 
Therefore, we can choose an element of the patch, say, $\widetilde T$ and send it to $\Tref$ such that its neighboring (transformed) elements 
\begin{align*}
  \widetilde T'' &=\conv\{(-\widetilde c/(\widetilde d-\widetilde c),1/(\widetilde d-\widetilde c))\}, \\
  \widetilde T' &=\conv\{(1/(\widetilde d'-\widetilde c'),-\widetilde c'/(\widetilde d'-\widetilde c'))\} \\
\end{align*}
satisfy $\widetilde d\neq 1$ or $\widetilde d'\neq 1$. For $q\in P^{n-1,n,\perp}(\widetilde T') \cap  P^{n-1,n,\perp}(\widetilde T) \cap P^{n-1,n,\perp}(\Tref)$ we conclude $q = 0$.
\qed

\end{document}

%% file: header.tex
\newtheorem{theorem}{Theorem}
\newtheorem{assumption}[theorem]{Assumption}
\newtheorem{lemma}[theorem]{Lemma}

\newtheorem{remark}[theorem]{Remark}

\newcommand{\patch}{\omega}
\newcommand{\Patch}{\Omega}
\newcommand{\tpatch}{\widetilde{\omega}}
\newcommand{\tPatch}{\widetilde{\Omega}}
\newcommand{\Tref}{\widehat{T}}
\newcommand{\Eref}{\widehat{E}}
\newcommand{\zref}{\widehat{z}}

\newcommand{\II}{\mathcal{I}}

\newcommand{\conv}{\operatorname{conv}}

\newcommand{\qintpz}{J_{0}^{p+1}}

\newcommand{\qintp}{J^{p+1}}

\newcommand{\qintzpz}{I_{0}^{p+1}}
\newcommand{\qintzp}{I^{p+1}}

\newcommand{\Pip}{\Pi^{p}}

\def\di{\mathrm{d}}
\DeclareMathOperator{\linhull}{span}

\newcommand{\ip}[2]{(#1\hspace*{.5mm},#2)}
\newcommand{\dual}[2]{\langle#1\hspace*{.5mm},#2\rangle}

\newcommand{\norm}[3][]{#1\|#2#1\|_{#3}}

\newcommand{\diam}{\mathrm{diam}}

\def\div{{\rm div\,}}

\newcommand{\Hdivset}[1]{\boldsymbol{H}(\div;#1)}

\newcommand{\set}[2]{\big\{#1\,:\,#2\big\}}

\newcommand{\RT}{\ensuremath{\boldsymbol{RT}}}

\newcommand{\R}{\ensuremath{\mathbb{R}}}
\newcommand{\N}{\ensuremath{\mathbb{N}}}

\newcommand{\LL}{\ensuremath{\boldsymbol{L}}}

\newcommand{\vv}{\ensuremath{\boldsymbol{v}}}
\newcommand{\ww}{\ensuremath{\boldsymbol{w}}}
\newcommand{\TT}{\ensuremath{\mathcal{T}}}

\newcommand{\OO}{\ensuremath{\mathcal{O}}}
\newcommand{\EE}{\ensuremath{\mathcal{E}}}

\newcommand{\normal}{\ensuremath{{\boldsymbol{n}}}}

\newcommand{\VV}{\ensuremath{\mathcal{V}}}

\newcommand{\ssigma}{{\boldsymbol\sigma}}
\newcommand{\ttau}{{\boldsymbol\tau}}

\newcommand{\qq}{{\boldsymbol{q}}}
\newcommand{\uu}{\boldsymbol{u}}

\newcommand{\rr}{{\boldsymbol{r}}}

\newcommand{\nulo}{0}
\definecolor{Azul}{HTML}{173F8A}

\newcommand{\logLogSlopeTriangle}[6]
{

    \pgfplotsextra
    {
        \pgfkeysgetvalue{/pgfplots/xmin}{\xmin}
        \pgfkeysgetvalue{/pgfplots/xmax}{\xmax}
        \pgfkeysgetvalue{/pgfplots/ymin}{\ymin}
        \pgfkeysgetvalue{/pgfplots/ymax}{\ymax}

        \pgfmathsetmacro{\xArel}{#1}
        \pgfmathsetmacro{\yArel}{#3}
        \pgfmathsetmacro{\xBrel}{#1-#2}
        \pgfmathsetmacro{\yBrel}{\yArel}
        \pgfmathsetmacro{\xCrel}{\xArel}

        \pgfmathsetmacro{\lnxB}{\xmin*(1-(#1-#2))+\xmax*(#1-#2)} 
        \pgfmathsetmacro{\lnxA}{\xmin*(1-#1)+\xmax*#1} 
        \pgfmathsetmacro{\lnyA}{\ymin*(1-#3)+\ymax*#3} 
        \pgfmathsetmacro{\lnyC}{\lnyA+#4*(\lnxA-\lnxB)}
        \pgfmathsetmacro{\yCrel}{\lnyC-\ymin)/(\ymax-\ymin)} 

        \coordinate (A) at (rel axis cs:\xArel,\yArel);
        \coordinate (B) at (rel axis cs:\xBrel,\yCrel);
        \coordinate (C) at (rel axis cs:\xCrel,\yCrel);

        \draw[#5]   (A)--
                    (B)-- 
                    (C)-- node[pos=0.5,anchor=west] {#6}
                    cycle;
    }
}

\newcommand{\logLogSlopeTriangleBelow}[6]
{

    \pgfplotsextra
    {
        \pgfkeysgetvalue{/pgfplots/xmin}{\xmin}
        \pgfkeysgetvalue{/pgfplots/xmax}{\xmax}
        \pgfkeysgetvalue{/pgfplots/ymin}{\ymin}
        \pgfkeysgetvalue{/pgfplots/ymax}{\ymax}

        \pgfmathsetmacro{\xArel}{#1}
        \pgfmathsetmacro{\yArel}{#3}
        \pgfmathsetmacro{\xBrel}{#1-#2}
        \pgfmathsetmacro{\yBrel}{\yArel}
        \pgfmathsetmacro{\xCrel}{\xArel}

        \pgfmathsetmacro{\lnxB}{\xmin*(1-(#1-#2))+\xmax*(#1-#2)} 
        \pgfmathsetmacro{\lnxA}{\xmin*(1-#1)+\xmax*#1} 
        \pgfmathsetmacro{\lnyA}{\ymin*(1-#3)+\ymax*#3} 
        \pgfmathsetmacro{\lnyC}{\lnyA+#4*(\lnxA-\lnxB)}
        \pgfmathsetmacro{\yCrel}{\lnyC-\ymin)/(\ymax-\ymin)} 

        \coordinate (A) at (rel axis cs:\xArel,\yArel);
        \coordinate (B) at (rel axis cs:\xBrel,\yCrel);
        \coordinate (C) at (rel axis cs:\xBrel,\yArel);

        \draw[#5]   (A)--
                    (B)-- node[pos=0.5,anchor=east] {#6}
                    (C)-- 
                    cycle;
    }
}

%% file: FigureMeshes.tex
\begin{tikzpicture}
\begin{axis}[hide axis,
    axis equal,
    width=0.45\textwidth,
axis equal image = true]

\addplot[patch,color=white,
faceted color = black, line width = 1.5pt,
patch table ={data/eleMeshAss1.dat}] file{data/cooMeshAss1.dat};
\addplot[mark=*,color=gray,only marks] table[x index=0,y index=1] {data/cooMeshAss1.dat};
\end{axis}
\end{tikzpicture}
\begin{tikzpicture}
\begin{axis}[hide axis,
    axis equal,
    width=0.45\textwidth,
axis equal image = true]

\addplot[patch,color=white,
faceted color = black, line width = 1.5pt,
patch table ={data/eleMeshAss2.dat}] file{data/cooMeshAss2.dat};
\addplot[mark=*,color=gray,only marks] table[x index=0,y index=1] {data/cooMeshAss2.dat};
\end{axis}
\end{tikzpicture}

%% file: FigurePatch1.tex
\begin{tikzpicture}[scale=4]

\coordinate (A0) at (0, 0.5);
\coordinate (A1) at (0.5, 0);
\coordinate (A2) at (1, 0.5);
\coordinate (A3) at (0.5, 1);
\coordinate (A4) at (0.0244716, 0.345492);
\coordinate (A5) at (0.0954907, 0.206109);
\coordinate (A6) at (0.206109, 0.0954907);
\coordinate (A7) at (0.345492, 0.0244716);
\coordinate (A8) at (0.654508, 0.0244716);
\coordinate (A9) at (0.793891, 0.0954907);
\coordinate (A10) at (0.904509, 0.206109);
\coordinate (A11) at (0.975528, 0.345492);
\coordinate (A12) at (0.975528, 0.654508);
\coordinate (A13) at (0.904509, 0.793891);
\coordinate (A14) at (0.793891, 0.904509);
\coordinate (A15) at (0.654508, 0.975528);
\coordinate (A16) at (0.345492, 0.975528);
\coordinate (A17) at (0.206109, 0.904509);
\coordinate (A18) at (0.0954907, 0.793891);
\coordinate (A19) at (0.0244716, 0.654508);
\coordinate (A20) at (0.159269, 0.379996);
\coordinate (A21) at (0.303367, 0.240319);
\coordinate (A22) at (0.486059, 0.147937);
\coordinate (A23) at (0.642521, 0.213929);
\coordinate (A24) at (0.791732, 0.289024);
\coordinate (A25) at (0.824893, 0.450791);
\coordinate (A26) at (0.845676, 0.613781);
\coordinate (A27) at (0.734247, 0.735212);
\coordinate (A28) at (0.613167, 0.846542);
\coordinate (A29) at (0.450184, 0.826664);
\coordinate (A30) at (0.289651, 0.795691);
\coordinate (A31) at (0.214897, 0.656412);
\coordinate (A32) at (0.145385, 0.533223);
\coordinate (A33) at (0.319046, 0.474615);
\coordinate (A34) at (0.687023, 0.565484);
\coordinate (A35) at (0.655041, 0.394412);
\coordinate (A36) at (0.482921, 0.331124);
\coordinate (A37) at (0.391363, 0.657008);
\coordinate (A38) at (0.56398, 0.68937);
\coordinate (A39) at (0.513911, 0.51673);

\fill[Azul!20] (A39) -- (A36) -- (A35) -- cycle;
\fill[Azul!20] (A39) -- (A35) -- (A34) -- cycle;
\fill[Azul!20] (A39) -- (A34) -- (A38) -- cycle;
\fill[Azul!20] (A39) -- (A38) -- (A37) -- cycle;
\fill[Azul!20] (A39) -- (A37) -- (A33) -- cycle;
\fill[Azul!20] (A39) -- (A33) -- (A36) -- cycle;

    \foreach \vertex in {A0, A1, A2, A3, A4, A5, A6, A7, A8, A9, A10, A11, A12, A13, A14, A15, A16, A17, A18, A19, A20, A21, A22, A23, A24, A25, A26, A27, A28, A29, A30, A31, A32, A33, A34, A35, A36, A37, A38}
        \fill[gray] (\vertex) circle (0.25pt);
    \fill[Azul] (A39) circle (0.75pt) node[right,xshift=.1cm, yshift=-0.075cm] {$z$};

\draw (A0) -- (A4);
\draw (A0) -- (A19);
\draw (A0) -- (A20);
\draw (A0) -- (A32);
\draw (A1) -- (A7);
\draw (A1) -- (A8);
\draw (A1) -- (A22);
\draw (A2) -- (A11);
\draw (A2) -- (A12);
\draw (A2) -- (A25);
\draw (A2) -- (A26);
\draw (A3) -- (A15);
\draw (A3) -- (A16);
\draw (A3) -- (A28);
\draw (A3) -- (A29);
\draw (A4) -- (A5);
\draw (A4) -- (A20);
\draw (A5) -- (A6);
\draw (A5) -- (A20);
\draw (A5) -- (A21);
\draw (A6) -- (A7);
\draw (A6) -- (A21);
\draw (A7) -- (A21);
\draw (A7) -- (A22);
\draw (A8) -- (A9);
\draw (A8) -- (A22);
\draw (A8) -- (A23);
\draw (A9) -- (A10);
\draw (A9) -- (A23);
\draw (A9) -- (A24);
\draw (A10) -- (A11);
\draw (A10) -- (A24);
\draw (A11) -- (A24);
\draw (A11) -- (A25);
\draw (A12) -- (A13);
\draw (A12) -- (A26);
\draw (A13) -- (A14);
\draw (A13) -- (A26);
\draw (A13) -- (A27);
\draw (A14) -- (A15);
\draw (A14) -- (A27);
\draw (A14) -- (A28);
\draw (A15) -- (A28);
\draw (A16) -- (A17);
\draw (A16) -- (A29);
\draw (A16) -- (A30);
\draw (A17) -- (A18);
\draw (A17) -- (A30);
\draw (A18) -- (A19);
\draw (A18) -- (A30);
\draw (A18) -- (A31);
\draw (A19) -- (A31);
\draw (A19) -- (A32);
\draw (A20) -- (A21);
\draw (A20) -- (A32);
\draw (A20) -- (A33);
\draw (A21) -- (A22);
\draw (A21) -- (A33);
\draw (A21) -- (A36);
\draw (A22) -- (A23);
\draw (A22) -- (A36);
\draw (A23) -- (A24);
\draw (A23) -- (A35);
\draw (A23) -- (A36);
\draw (A24) -- (A25);
\draw (A24) -- (A35);
\draw (A25) -- (A26);
\draw (A25) -- (A34);
\draw (A25) -- (A35);
\draw (A26) -- (A27);
\draw (A26) -- (A34);
\draw (A27) -- (A28);
\draw (A27) -- (A34);
\draw (A27) -- (A38);
\draw (A28) -- (A29);
\draw (A28) -- (A38);
\draw (A29) -- (A30);
\draw (A29) -- (A37);
\draw (A29) -- (A38);
\draw (A30) -- (A31);
\draw (A30) -- (A37);
\draw (A31) -- (A32);
\draw (A31) -- (A33);
\draw (A31) -- (A37);
\draw (A32) -- (A33);
\draw (A33) -- (A36);
\draw (A33) -- (A37);
\draw (A33) -- (A39);
\draw (A34) -- (A35);
\draw (A34) -- (A38);
\draw (A34) -- (A39);
\draw (A35) -- (A36);
\draw (A35) -- (A39);
\draw (A36) -- (A39);
\draw (A37) -- (A38);
\draw (A37) -- (A39);
\draw (A38) -- (A39);

\end{tikzpicture}

%% file: FigurePatch2.tex
\begin{tikzpicture}[scale=4]
\coordinate (A0) at (0, 0.5);
\coordinate (A1) at (0.5, 0);
\coordinate (A2) at (1, 0.5);
\coordinate (A3) at (0.5, 1);
\coordinate (A4) at (0.0244716, 0.345492);
\coordinate (A5) at (0.0954907, 0.206109);
\coordinate (A6) at (0.206109, 0.0954907);
\coordinate (A7) at (0.345492, 0.0244716);
\coordinate (A8) at (0.654508, 0.0244716);
\coordinate (A9) at (0.793891, 0.0954907);
\coordinate (A10) at (0.904509, 0.206109);
\coordinate (A11) at (0.975528, 0.345492);
\coordinate (A12) at (0.975528, 0.654508);
\coordinate (A13) at (0.904509, 0.793891);
\coordinate (A14) at (0.793891, 0.904509);
\coordinate (A15) at (0.654508, 0.975528);
\coordinate (A16) at (0.345492, 0.975528);
\coordinate (A17) at (0.206109, 0.904509);
\coordinate (A18) at (0.0954907, 0.793891);
\coordinate (A19) at (0.0244716, 0.654508);
\coordinate (A20) at (0.159269, 0.379996);
\coordinate (A21) at (0.303367, 0.240319);
\coordinate (A22) at (0.486059, 0.147937);
\coordinate (A23) at (0.642521, 0.213929);
\coordinate (A24) at (0.791732, 0.289024);
\coordinate (A25) at (0.824893, 0.450791);
\coordinate (A26) at (0.845676, 0.613781);
\coordinate (A27) at (0.734247, 0.735212);
\coordinate (A28) at (0.613167, 0.846542);
\coordinate (A29) at (0.450184, 0.826664);
\coordinate (A30) at (0.289651, 0.795691);
\coordinate (A31) at (0.214897, 0.656412);
\coordinate (A32) at (0.145385, 0.533223);
\coordinate (A33) at (0.319046, 0.474615);
\coordinate (A34) at (0.687023, 0.565484);
\coordinate (A35) at (0.655041, 0.394412);
\coordinate (A36) at (0.482921, 0.331124);
\coordinate (A37) at (0.391363, 0.657008);
\coordinate (A38) at (0.56398, 0.68937);
\coordinate (A39) at (0.513911, 0.51673);

\fill[Azul!20] (A39) -- (A36) -- (A35) -- cycle;
\fill[Azul!20] (A39) -- (A35) -- (A34) -- cycle;
\fill[Azul!20] (A39) -- (A34) -- (A38) -- cycle;
\fill[Azul!20] (A39) -- (A38) -- (A37) -- cycle;
\fill[Azul!20] (A39) -- (A37) -- (A33) -- cycle;
\fill[Azul!20] (A39) -- (A33) -- (A36) -- cycle;

\fill[Azul!20] (A36) -- (A33) -- (A21) -- cycle;
\fill[Azul!20] (A36) -- (A21) -- (A22) -- cycle;
\fill[Azul!20] (A36) -- (A22) -- (A23) -- cycle;
\fill[Azul!20] (A36) -- (A23) -- (A35) -- cycle;

\fill[Azul!20] (A35) -- (A23) -- (A24) -- cycle;
\fill[Azul!20] (A35) -- (A24) -- (A25) -- cycle;
\fill[Azul!20] (A35) -- (A25) -- (A34) -- cycle;

\fill[Azul!20] (A34) -- (A25) -- (A26) -- cycle;
\fill[Azul!20] (A34) -- (A26) -- (A27) -- cycle;
\fill[Azul!20] (A34) -- (A27) -- (A38) -- cycle;

\fill[Azul!20] (A38) -- (A27) -- (A28) -- cycle;
\fill[Azul!20] (A38) -- (A28) -- (A29) -- cycle;
\fill[Azul!20] (A38) -- (A29) -- (A37) -- cycle;

\fill[Azul!20] (A37) -- (A29) -- (A30) -- cycle;
\fill[Azul!20] (A37) -- (A30) -- (A31) -- cycle;
\fill[Azul!20] (A37) -- (A31) -- (A33) -- cycle;

\fill[Azul!20] (A33) -- (A31) -- (A32) -- cycle;
\fill[Azul!20] (A33) -- (A32) -- (A20) -- cycle;
\fill[Azul!20] (A33) -- (A20) -- (A21) -- cycle;
    \foreach \vertex in {A0, A1, A2, A3, A4, A5, A6, A7, A8, A9, A10, A11, A12, A13, A14, A15, A16, A17, A18, A19, A20, A21, A22, A23, A24, A25, A26, A27, A28, A29, A30, A31, A32, A33, A34, A35, A36, A37, A38}
        \fill[gray] (\vertex) circle (0.25pt);
    \fill[Azul] (A39) circle (0.75pt) node[right,xshift=.1cm, yshift=-0.075cm] {$z$};


\draw (A0) -- (A4);
\draw (A0) -- (A19);
\draw (A0) -- (A20);
\draw (A0) -- (A32);
\draw (A1) -- (A7);
\draw (A1) -- (A8);
\draw (A1) -- (A22);
\draw (A2) -- (A11);
\draw (A2) -- (A12);
\draw (A2) -- (A25);
\draw (A2) -- (A26);
\draw (A3) -- (A15);
\draw (A3) -- (A16);
\draw (A3) -- (A28);
\draw (A3) -- (A29);
\draw (A4) -- (A5);
\draw (A4) -- (A20);
\draw (A5) -- (A6);
\draw (A5) -- (A20);
\draw (A5) -- (A21);
\draw (A6) -- (A7);
\draw (A6) -- (A21);
\draw (A7) -- (A21);
\draw (A7) -- (A22);
\draw (A8) -- (A9);
\draw (A8) -- (A22);
\draw (A8) -- (A23);
\draw (A9) -- (A10);
\draw (A9) -- (A23);
\draw (A9) -- (A24);
\draw (A10) -- (A11);
\draw (A10) -- (A24);
\draw (A11) -- (A24);
\draw (A11) -- (A25);
\draw (A12) -- (A13);
\draw (A12) -- (A26);
\draw (A13) -- (A14);
\draw (A13) -- (A26);
\draw (A13) -- (A27);
\draw (A14) -- (A15);
\draw (A14) -- (A27);
\draw (A14) -- (A28);
\draw (A15) -- (A28);
\draw (A16) -- (A17);
\draw (A16) -- (A29);
\draw (A16) -- (A30);
\draw (A17) -- (A18);
\draw (A17) -- (A30);
\draw (A18) -- (A19);
\draw (A18) -- (A30);
\draw (A18) -- (A31);
\draw (A19) -- (A31);
\draw (A19) -- (A32);
\draw (A20) -- (A21);
\draw (A20) -- (A32);
\draw (A20) -- (A33);
\draw (A21) -- (A22);
\draw (A21) -- (A33);
\draw (A21) -- (A36);
\draw (A22) -- (A23);
\draw (A22) -- (A36);
\draw (A23) -- (A24);
\draw (A23) -- (A35);
\draw (A23) -- (A36);
\draw (A24) -- (A25);
\draw (A24) -- (A35);
\draw (A25) -- (A26);
\draw (A25) -- (A34);
\draw (A25) -- (A35);
\draw (A26) -- (A27);
\draw (A26) -- (A34);
\draw (A27) -- (A28);
\draw (A27) -- (A34);
\draw (A27) -- (A38);
\draw (A28) -- (A29);
\draw (A28) -- (A38);
\draw (A29) -- (A30);
\draw (A29) -- (A37);
\draw (A29) -- (A38);
\draw (A30) -- (A31);
\draw (A30) -- (A37);
\draw (A31) -- (A32);
\draw (A31) -- (A33);
\draw (A31) -- (A37);
\draw (A32) -- (A33);
\draw (A33) -- (A36);
\draw (A33) -- (A37);
\draw (A33) -- (A39);
\draw (A34) -- (A35);
\draw (A34) -- (A38);
\draw (A34) -- (A39);
\draw (A35) -- (A36);
\draw (A35) -- (A39);
\draw (A36) -- (A39);
\draw (A37) -- (A38);
\draw (A37) -- (A39);
\draw (A38) -- (A39);

\end{tikzpicture}

%% file: FigurePatch3.tex
\begin{tikzpicture}[scale=4]
\coordinate (A0) at (0, 0.5);
\coordinate (A1) at (0.5, 0);
\coordinate (A2) at (1, 0.5);
\coordinate (A3) at (0.5, 1);
\coordinate (A4) at (0.0244716, 0.345492);
\coordinate (A5) at (0.0954907, 0.206109);
\coordinate (A6) at (0.206109, 0.0954907);
\coordinate (A7) at (0.345492, 0.0244716);
\coordinate (A8) at (0.654508, 0.0244716);
\coordinate (A9) at (0.793891, 0.0954907);
\coordinate (A10) at (0.904509, 0.206109);
\coordinate (A11) at (0.975528, 0.345492);
\coordinate (A12) at (0.975528, 0.654508);
\coordinate (A13) at (0.904509, 0.793891);
\coordinate (A14) at (0.793891, 0.904509);
\coordinate (A15) at (0.654508, 0.975528);
\coordinate (A16) at (0.345492, 0.975528);
\coordinate (A17) at (0.206109, 0.904509);
\coordinate (A18) at (0.0954907, 0.793891);
\coordinate (A19) at (0.0244716, 0.654508);
\coordinate (A20) at (0.159269, 0.379996);
\coordinate (A21) at (0.303367, 0.240319);
\coordinate (A22) at (0.486059, 0.147937);
\coordinate (A23) at (0.642521, 0.213929);
\coordinate (A24) at (0.791732, 0.289024);
\coordinate (A25) at (0.824893, 0.450791);
\coordinate (A26) at (0.845676, 0.613781);
\coordinate (A27) at (0.734247, 0.735212);
\coordinate (A28) at (0.613167, 0.846542);
\coordinate (A29) at (0.450184, 0.826664);
\coordinate (A30) at (0.289651, 0.795691);
\coordinate (A31) at (0.214897, 0.656412);
\coordinate (A32) at (0.145385, 0.533223);
\coordinate (A33) at (0.319046, 0.474615);
\coordinate (A34) at (0.687023, 0.565484);
\coordinate (A35) at (0.655041, 0.394412);
\coordinate (A36) at (0.482921, 0.331124);
\coordinate (A37) at (0.391363, 0.657008);
\coordinate (A38) at (0.56398, 0.68937);
\coordinate (A39) at (0.513911, 0.51673);

\fill[Azul!20] (A0) -- (A4) -- (A5) -- (A6) -- (A7) -- (A1) -- (A8) -- (A9) -- (A10) -- (A11) -- (A2) -- (A12) -- (A13) -- (A14) -- (A15) -- (A3) -- (A16) -- (A17) -- (A18) -- (A19) -- (A0)-- cycle;

    \foreach \vertex in {A0, A1, A2, A3, A4, A5, A6, A7, A8, A9, A10, A11, A12, A13, A14, A15, A16, A17, A18, A19, A20, A21, A22, A23, A24, A25, A26, A27, A28, A29, A30, A31, A32, A33, A34, A35, A36, A37, A38}
        \fill[gray] (\vertex) circle (0.25pt);
    \fill[Azul] (A39) circle (0.75pt) node[right,xshift=.1cm, yshift=-0.075cm] {$z$};


\draw (A0) -- (A4);
\draw (A0) -- (A19);
\draw (A0) -- (A20);
\draw (A0) -- (A32);
\draw (A1) -- (A7);
\draw (A1) -- (A8);
\draw (A1) -- (A22);
\draw (A2) -- (A11);
\draw (A2) -- (A12);
\draw (A2) -- (A25);
\draw (A2) -- (A26);
\draw (A3) -- (A15);
\draw (A3) -- (A16);
\draw (A3) -- (A28);
\draw (A3) -- (A29);
\draw (A4) -- (A5);
\draw (A4) -- (A20);
\draw (A5) -- (A6);
\draw (A5) -- (A20);
\draw (A5) -- (A21);
\draw (A6) -- (A7);
\draw (A6) -- (A21);
\draw (A7) -- (A21);
\draw (A7) -- (A22);
\draw (A8) -- (A9);
\draw (A8) -- (A22);
\draw (A8) -- (A23);
\draw (A9) -- (A10);
\draw (A9) -- (A23);
\draw (A9) -- (A24);
\draw (A10) -- (A11);
\draw (A10) -- (A24);
\draw (A11) -- (A24);
\draw (A11) -- (A25);
\draw (A12) -- (A13);
\draw (A12) -- (A26);
\draw (A13) -- (A14);
\draw (A13) -- (A26);
\draw (A13) -- (A27);
\draw (A14) -- (A15);
\draw (A14) -- (A27);
\draw (A14) -- (A28);
\draw (A15) -- (A28);
\draw (A16) -- (A17);
\draw (A16) -- (A29);
\draw (A16) -- (A30);
\draw (A17) -- (A18);
\draw (A17) -- (A30);
\draw (A18) -- (A19);
\draw (A18) -- (A30);
\draw (A18) -- (A31);
\draw (A19) -- (A31);
\draw (A19) -- (A32);
\draw (A20) -- (A21);
\draw (A20) -- (A32);
\draw (A20) -- (A33);
\draw (A21) -- (A22);
\draw (A21) -- (A33);
\draw (A21) -- (A36);
\draw (A22) -- (A23);
\draw (A22) -- (A36);
\draw (A23) -- (A24);
\draw (A23) -- (A35);
\draw (A23) -- (A36);
\draw (A24) -- (A25);
\draw (A24) -- (A35);
\draw (A25) -- (A26);
\draw (A25) -- (A34);
\draw (A25) -- (A35);
\draw (A26) -- (A27);
\draw (A26) -- (A34);
\draw (A27) -- (A28);
\draw (A27) -- (A34);
\draw (A27) -- (A38);
\draw (A28) -- (A29);
\draw (A28) -- (A38);
\draw (A29) -- (A30);
\draw (A29) -- (A37);
\draw (A29) -- (A38);
\draw (A30) -- (A31);
\draw (A30) -- (A37);
\draw (A31) -- (A32);
\draw (A31) -- (A33);
\draw (A31) -- (A37);
\draw (A32) -- (A33);
\draw (A33) -- (A36);
\draw (A33) -- (A37);
\draw (A33) -- (A39);
\draw (A34) -- (A35);
\draw (A34) -- (A38);
\draw (A34) -- (A39);
\draw (A35) -- (A36);
\draw (A35) -- (A39);
\draw (A36) -- (A39);
\draw (A37) -- (A38);
\draw (A37) -- (A39);
\draw (A38) -- (A39);

\end{tikzpicture}

%% file: PostprocessingOperator.bbl
\begin{thebibliography}{FHK22}

\bibitem[BBF13]{BoffiBrezziFortin}
Daniele Boffi, Franco Brezzi, and Michel Fortin.
\newblock {\em Mixed finite element methods and applications}, volume~44 of
  {\em Springer Series in Computational Mathematics}.
\newblock Springer, Heidelberg, 2013.

\bibitem[BG98]{BernardiGirault98}
C.~Bernardi and V.~Girault.
\newblock A local regularization operator for triangular and quadrilateral
  finite elements.
\newblock {\em SIAM J. Numer. Anal.}, 35(5):1893--1916, 1998.

\bibitem[BPS13]{BankParsaniaSauter13}
Randolph~E. Bank, Asieh Parsania, and Stefan Sauter.
\newblock Saturation estimates for {$hp$}-finite element methods.
\newblock {\em Comput. Vis. Sci.}, 16(5):195--217, 2013.

\bibitem[BS08]{BrennerScott}
Susanne~C. Brenner and L.~Ridgway Scott.
\newblock {\em The mathematical theory of finite element methods}, volume~15 of
  {\em Texts in Applied Mathematics}.
\newblock Springer, New York, third edition, 2008.

\bibitem[BS22]{SupConvDPGelasticity}
Fleurianne Bertrand and Henrik Schneider.
\newblock Superconvergence of discontinuous petrov-galerkin approximations in
  linear elasticity.
\newblock {\em arXiv}, arXiv:2209.08859, 2022.

\bibitem[BX89]{BrambleXu89}
James~H. Bramble and Jinchao Xu.
\newblock A local post-processing technique for improving the accuracy in mixed
  finite-element approximations.
\newblock {\em SIAM J. Numer. Anal.}, 26(6):1267--1275, 1989.

\bibitem[Car99]{CarstensenQuasiInt99}
Carsten Carstensen.
\newblock Quasi-interpolation and a posteriori error analysis in finite element
  methods.
\newblock {\em M2AN Math. Model. Numer. Anal.}, 33(6):1187--1202, 1999.

\bibitem[CGS10]{CockburnGopalakrishnanSayas10}
Bernardo Cockburn, Jayadeep Gopalakrishnan, and Francisco-Javier Sayas.
\newblock A projection-based error analysis of {HDG} methods.
\newblock {\em Math. Comp.}, 79(271):1351--1367, 2010.

\bibitem[Cl{\'{e}}75]{Clement75}
Ph. Cl{\'{e}}ment.
\newblock Approximation by finite element functions using local regularization.
\newblock {\em Rev. Fran\c{c}aise Automat. Informat. Recherche
  Op\'{e}rationnelle S\'{e}r.}, 9({\rm R}-2):77--84, 1975.

\bibitem[DST23]{DieningStornTscherpel23}
Lars Diening, Johannes Storn, and Tabea Tscherpel.
\newblock Interpolation operator on negative {S}obolev spaces.
\newblock {\em Math. Comp.}, 92(342):1511--1541, 2023.

\bibitem[EG17]{ErnGuermond17}
Alexandre Ern and Jean-Luc Guermond.
\newblock Finite element quasi-interpolation and best approximation.
\newblock {\em ESAIM Math. Model. Numer. Anal.}, 51(4):1367--1385, 2017.

\bibitem[FHK22]{FHK22}
Thomas F\"{u}hrer, Norbert Heuer, and Michael Karkulik.
\newblock M{INRES} for second-order {PDE}s with singular data.
\newblock {\em SIAM J. Numer. Anal.}, 60(3):1111--1135, 2022.

\bibitem[F{\"{u}}h19]{SupConv2}
Thomas F{\"{u}}hrer.
\newblock Superconvergent {DPG} methods for second-order elliptic problems.
\newblock {\em Comput. Methods Appl. Math.}, 19(3):483--502, 2019.

\bibitem[F{\"{u}}h21]{MultilevelNorms21}
Thomas F{\"{u}}hrer.
\newblock Multilevel decompositions and norms for negative order {S}obolev
  spaces.
\newblock {\em Math. Comp.}, 91(333):183--218, 2021.

\bibitem[F{\"{u}}h23]{MixedFEMHm1}
Thomas F{\"{u}}hrer.
\newblock On a mixed {FEM} and a {FOSLS} with {$H^{-1}$} loads.
\newblock {\em Comput. Methods Appl. Math.}, published online, 2023.

\bibitem[KS15]{KoornwinderSauter15}
Tom~H. Koornwinder and Stefan~A. Sauter.
\newblock The intersection of bivariate orthogonal polynomials on triangle
  patches.
\newblock {\em Math. Comp.}, 84(294):1795--1812, 2015.

\bibitem[Mel05]{Melenk05}
J.~M. Melenk.
\newblock {$hp$}-interpolation of nonsmooth functions and an application to
  {$hp$}-a posteriori error estimation.
\newblock {\em SIAM J. Numer. Anal.}, 43(1):127--155, 2005.

\bibitem[Osw93]{Oswald93}
P.~Oswald.
\newblock On a {BPX}-preconditioner for {${\rm P}1$} elements.
\newblock {\em Computing}, 51(2):125--133, 1993.

\bibitem[{\v{S}}SD04]{HighOrderFEM}
Pavel {\v{S}}ol{\'{\i}}n, Karel Segeth, and Ivo Dole\v{z}el.
\newblock {\em Higher-order finite element methods}.
\newblock Studies in Advanced Mathematics. Chapman \& Hall/CRC, Boca Raton, FL,
  2004.
\newblock With 1 CD-ROM (Windows, Macintosh, UNIX and LINUX).

\bibitem[Ste91]{Stenberg91}
Rolf Stenberg.
\newblock Postprocessing schemes for some mixed finite elements.
\newblock {\em RAIRO Mod\'{e}l. Math. Anal. Num\'{e}r.}, 25(1):151--167, 1991.

\bibitem[SvV20]{StevensonVanVenetie20}
Rob Stevenson and Raymond van Veneti\"{e}.
\newblock Uniform preconditioners for problems of negative order.
\newblock {\em Math. Comp.}, 89(322):645--674, 2020.

\bibitem[SZ90]{SZ_90}
L.~Ridgway Scott and Shangyou Zhang.
\newblock Finite element interpolation of nonsmooth functions satisfying
  boundary conditions.
\newblock {\em Math. Comp.}, 54(190):483--493, 1990.

\bibitem[VZ18]{VeeserZanotti18}
Andreas Veeser and Pietro Zanotti.
\newblock Quasi-optimal nonconforming methods for symmetric elliptic problems.
  {I}---{A}bstract theory.
\newblock {\em SIAM J. Numer. Anal.}, 56(3):1621--1642, 2018.

\end{thebibliography}
